\documentclass{amsart}
\usepackage{amssymb}
\usepackage[nobysame]{amsrefs}
\usepackage{todonotes}
\usepackage{mathpazo}
\usepackage{enumerate}
\usepackage{bm}
\usepackage{tikz,pgf,pgfmath}
	\pgfdeclarelayer{background}
	\pgfdeclarelayer{middle}
	\pgfsetlayers{background,middle,main}
\usepackage[colorlinks=true,
	urlcolor=blue!20!black,
	citecolor=green!50!black]{hyperref}
\definecolor{c1}{HTML}{FF9999}
\definecolor{c2}{HTML}{FFD27F}
\definecolor{c3}{HTML}{99C199}
\definecolor{c4}{HTML}{9999FF}
\definecolor{c5}{HTML}{99D1D1}
\definecolor{c6}{HTML}{D9A6F9}
\definecolor{c7}{HTML}{FFA500}
\newtheorem{theorem}{Theorem}[section]
\newtheorem{proposition}[theorem]{Proposition}

\newtheorem{lemma}[theorem]{Lemma}
\newtheorem{corollary}[theorem]{Corollary}
\theoremstyle{remark}
\newtheorem{definition}[theorem]{Definition}
\newtheorem{construction}[theorem]{Construction}

\newtheorem{remark}[theorem]{Remark}

\newcommand{\defn}[1]{{\color{green!50!black}\emph{#1}}}
\newcommand{\defs}{\stackrel{\mathrm{def}}{=}}
\newcommand{\ie}{{\it{i.e.}},\;}
\renewcommand{\th}{^{\text{th}}}

\newcommand{\Tamari}{\mathcal{T}}
\newcommand{\indicator}{\varrho}
\newcommand{\Symmetric}{\mathfrak{S}}
\newcommand{\weakorder}{\leq_{L}}
\newcommand{\comporder}{\leq_{\mathsf{comp}}}
\newcommand{\Codes}{\mathcal{C}}
\newcommand{\code}{\mathrm{code}}
\newcommand{\proj}{\pi^{\downarrow}}
\newcommand{\Inv}{\mathrm{Inv}}

\newcommand{\bvec}{\mathbf{b}}
\newcommand{\cvec}{\mathbf{c}}
\newcommand{\rvec}{\mathbf{r}}
\newcommand{\bvecset}{\mathcal{B}}
\newcommand{\rvecset}{\mathcal{R}}
\newcommand{\redbvec}{\Lambda_{\mathrm{red}}}
\newcommand{\extrvec}{\Lambda_{\mathrm{ext}}}
\newcommand{\codetorvec}{\Gamma_{\rvecset}}
\newcommand{\rvectocode}{\Delta_{\rvecset}}

\newcommand{\ddyck}{\mathbf{Dyck}}

\newcommand{\conj}{\mathbf{Conj}}
\newcommand{\flip}{\mathbf{Flip}}
\newcommand{\drun}{\mathbf{Drun}}
\newcommand{\cont}{\mathbf{Cont}}

\newcommand{\horiz}{\mathrm{horiz}}
\newcommand{\Paths}{\mathcal{L}}
\newcommand{\Dyck}{\mathcal{D}}

\BibSpec{collection.article}{%
    +{}  {\PrintAuthors}                {author}
    +{,} { \textit}                     {title}
    +{.} { In: }                        {booktitle}
    +{}  { \PrintEditorsB}              {editor}
    +{,}  { }                           {publisher}
    +{,}  { }                           {address}
    +{} { \parenthesize}                {date}
    +{,} { }                            {pages}
    +{.}  {\PrintReviews} {review}
}
\title{A Consecutive Lehmer Code for Parabolic Quotients of the Symmetric Group}
\author{Wenjie Fang}
	\address{WF: LIGM, Univ. Gustave Eiffel, CNRS, ESIEE Paris, F-77454 Marne-la-Vallée, France}
	\email{wenjie.fang@u-pem.fr}
\author{Henri M{\"u}hle}
	\address{HM: Technische Universit{\"a}t Dresden, Institut f{\"u}r Algebra, Zellescher Weg 12--14, 01069 Dresden, Germany.}
	\email{henri.muehle@tu-dresden.de}
\author{Jean-Christophe Novelli}
	\address{JCN: LIGM, Univ. Gustave Eiffel, CNRS, ESIEE Paris, F-77454 Marne-la-Vallée, France}
	\email{novelli@univ-mlv.fr}
\newcounter{i}
\newcounter{n}

\makeatletter
\let\orgdescriptionlabel\descriptionlabel
\renewcommand*{\descriptionlabel}[1]{%
  \let\orglabel\label
  \let\label\@gobble
  \phantomsection
  \protected@edef\@currentlabel{#1\unskip}%
  \let\label\orglabel
  \orgdescriptionlabel{(#1)}%
}
\makeatother

\newcommand{\AlphaPerm}[5]{
	\edef\alphaList{{#1}}
	\begin{tikzpicture}
		\def\d{#5};
		\def\dx{#4};
		\setcounter{i}{0}
		\setcounter{n}{0}
		\foreach \cl in {#2}{
			\pgfmathsetmacro{\ax}{\alphaList[\thei]}
			\begin{pgfonlayer}{background}
				\fill[\cl]({(\then+1-1.5*\d)*\dx},-2*\d*\dx) -- ({(\then+\ax+1.5*\d)*\dx},-2*\d*\dx) -- ({(\then+\ax+1.5*\d)*\dx},2*\d*\dx) -- ({(\then+1-1.5*\d)*\dx},2*\d*\dx) -- cycle;
			\end{pgfonlayer}
			\addtocounter{n}{\ax}
			\stepcounter{i}
		}
		\setcounter{i}{1}
		\foreach \k in {#3}{
			\draw(\thei*\dx,0) node{\k};
			\addtocounter{i}{1};
		}
	\end{tikzpicture}
}

\begin{document}

\begin{abstract}
In this article we define an encoding for parabolic permutations that
distinguishes between parabolic $231$-avoiding permutations. We prove that
the componentwise order on these codes realizes the parabolic Tamari lattice,
and conclude a direct and simple proof that the parabolic Tamari lattice is
isomorphic to a certain $\nu$-Tamari lattice, with an explicit bijection. Furthermore, we prove that this bijection is closely related to the map $\Theta$ used when the lattice isomorphism was first proved in (Ceballos, Fang and M{\"u}hle, 2020), settling an open problem therein.
\end{abstract}

\maketitle

\section{Introduction}
\label{sec:introduction}

A (right) \defn{inversion} of a permutation $w$ is a pair of indices $(i,j)$
with $i<j$ such that $w(i)>w(j)$. The number of inversions of $w$ can
therefore be regarded as a degree of disorder of $w$.
The \defn{Lehmer code} associated with $w$ is the integer tuple whose $i\th$
entry counts the number of inversions of $w$ of the form
$(i,\cdot)$~\cites{laisant88sur,lehmer60teaching}.  

Bj{\"o}rner and Wachs defined a ``consecutive'' version of the Lehmer code in
\cite{bjorner97shellable}*{Section~9}, which we shall call the \defn{BW-code}
of $w$. This encoding associates an integer tuple with a permutation $w$
whose $i\th$ entry counts the length $k$ of the longest sequence such that
$(i,j)$ is an inversion for all $j\in\{i{+}1,i{+}2,\ldots,i{+}k\}$.

In contrast to the original Lehmer code, the BW-code no longer uniquely
determines a permutation. However, the permutations with the same BW-code
form an interval in the (left) weak order on the group of all permutations,
the \defn{symmetric group}~\cite{bjorner97shellable}*{Proposition~9.10}. This
(left) \defn{weak order} is defined by containment of (right) inversion sets.

Another consequence of \cite{bjorner97shellable}*{Proposition~9.10} is that
among all permutations with the same BW-code, there is a unique permutation
$w$ which \defn{avoids} the pattern $231$, \ie in which no three indices
$i<j<k$ exist such that $w(k)<w(i)<w(j)$, and this permutation minimizes the
number of inversions among all permutations with the same BW-code as $w$.

Let us denote the symmetric group of degree $n$ by $\Symmetric_{n}$, and its
subset of all $231$-avoiding permutations by $\Symmetric_{n}(231)$. The
(left) weak order on $\Symmetric_{n}$ is a lattice, \ie every two elements
have a unique lower bound and a unique upper
bound~\cites{guilbaud71analyse,yanagimoto69partial}.
The restriction of this lattice to $\Symmetric_{n}(231)$ constitutes a
sublattice~\cite{bjorner97shellable}*{Theorem~9.6(i)} and a quotient
lattice~\cite{reading06cambrian}*{Theorem~5.1}. In fact, the resulting
lattice incarnates the famous \defn{Tamari lattice} denoted by
$\Tamari_{n}$~\cite{tamari62algebra}. We can thus see the BW-code as a
concrete and simple way to quotient the weak order on $\Symmetric_{n}$ into
$\Tamari_{n}$.

An analogue of 231-avoiding permutations for parabolic quotients of
$\Symmetric_{n}$ was introduced in \cite{muehle19tamari}, and it was shown
that these permutations constitute a quotient lattice (but no longer a
sublattice) of the corresponding (left) weak order, the \defn{parabolic Tamari
lattice}~\cite{muehle19tamari}*{Theorem~1}. Since any parabolic quotient of
$\Symmetric_{n}$ is naturally indexed by a composition $\alpha$ of $n$, we
call the resulting lattice the \defn{$\alpha$-Tamari lattice}
$\Tamari_{\alpha}$. The main purpose of this article is to define a parabolic
analogue of the BW-code; see Definition~\ref{def:all_codes}. We prove that the
componentwise order on these parabolic BW-codes is isomorphic to
$\Tamari_{\alpha}$.

Let us denote the set of parabolic BW-codes by $\Codes_{\alpha}$, and let us
denote the componentwise order on integer tuples (of the same length) by
$\comporder$. Our first main result now reads as follows.

\begin{theorem}
\label{thm:alpha_tamari_lattice_codes}
For every $n>0$ and every integer composition $\alpha$ of $n$ it holds that
$\Tamari_{\alpha}\cong (\Codes_{\alpha},\comporder)$.
\end{theorem}

Originally, the Tamari lattice was defined in terms of a ``rotation''
operation on parenthesizations, binary trees or equivalently Dyck paths. A
\defn{northeast path} is a lattice path in $\mathbb{N}^{2}$ comprised of
north steps (marked by $N$) and east steps (marked by $E$) of unit length. A
\defn{Dyck path} of semilength $n$ is equivalent to a northeast path that stays weakly above
the staircase path $(NE)^{n}$ and uses $n$ north and $n$ east steps.

A \defn{rotation} of a northeast path exchanges two portions of the path under
certain conditions, and $\Tamari_{n}$ arises as the rotation order on the set
of Dyck paths of semilength $n$. An extension of this construction was
introduced in \cite{preville17enumeration}. In that paper, the set of all
northeast paths weakly above a fixed northeast path $\nu$, which start and end
at the same coordinates as $\nu$, was considered. Ordering this set by
rotation produces another lattice, the \defn{$\nu$-Tamari
lattice}~\cite{preville17enumeration}*{Theorem~1.1}.  

For any composition $\alpha=(\alpha_{1},\alpha_{2},\ldots,\alpha_{r})$ of $n$,
we can define the $\alpha$-bounce path
$\nu_{\alpha}=N^{\alpha_{1}}E^{\alpha_{1}}N^{\alpha_{2}}E^{\alpha_{2}}\cdots
N^{\alpha_{r}}E^{\alpha_{r}}$. Theorem~II in \cite{ceballos20the} established
that $\Tamari_{\alpha}$ is isomorphic to the $\nu_{\alpha}$-Tamari lattice.
The proof of this result is rather technical, using some deep
lattice-theoretic properties of $\Tamari_{\alpha}$. The second main
contribution of this article is a much simpler and direct proof of this
result.

In general, the $\nu$-Tamari lattice admits a simple encoding as the
componentwise order on so-called $\nu$-bracket
vectors~\cite{ceballos20the}*{Theorem~4.2}. If $\nu=\nu_{\alpha}$, then the
corresponding bracket vectors can be converted in a simple way into parabolic
BW-codes. Since both parabolic BW-codes and bracket vectors are ordered
componentwise, the proof of the next result follows readily.

\begin{theorem}[\cite{ceballos20the}*{Theorem~II}]
\label{thm:alpha_tamari_isomorphism}
For every $n>0$ and every integer composition $\alpha$ of $n$, the
$\nu_{\alpha}$-Tamari lattice is isomorphic to $\Tamari_{\alpha}$.
\end{theorem}

The original proof of Theorem~\ref{thm:alpha_tamari_isomorphism} in
\cite{ceballos20the} did not provide an explicit map between the two lattices,
but rather passed through their Galois graphs, whose elements are related by a
map $\Theta$ between the two lattices. In \cite{ceballos20the}, it was
postulated as Open Problem~2.23 to prove that $\Theta$ extends to a full
lattice isomorphism, not limited to elements of the Galois graphs. Using
parabolic BW-codes, by introducing a stack processing procedure on
$(\alpha,231)$-avoiding permutations, we settle this open problem
affirmatively, while giving another interpretation of parabolic BW-codes; see Corollary~\ref{coro:phi-theta}.

\medskip

In Section~\ref{sec:basics}, we recall the basic definitions regarding
parabolic quotients of the symmetric group, parabolic pattern avoidance and
the weak order. In Section~\ref{sec:alpha_codes}, we define the parabolic
BW-codes and prove Theorem~\ref{thm:alpha_tamari_lattice_codes}. 

In Section~\ref{sec:lattice_paths}, we recall the definitions of Dyck paths
and northeast paths, as well as ordinary Tamari lattices and $\nu$-Tamari
lattices.
We then sketch a proof of the fact that $\nu$-Tamari lattices are intervals of an ordinary Tamari lattice using a map that has appeared in \cite{preville17enumeration}.  While this map was originally described in terms of trees, we take the perspective of Dyck paths, which makes its description much simpler; see Section~\ref{sec:nu_tamari_intervals}.
In Section~\ref{sec:dyck_sequence_stats}, we use the Dyck path-language to review that the anti-isomorphism on the Tamari lattice exchanges two sequence statistics on Dyck paths.  This property is well-known to experts, but only appears in print in the more general framework of Tamari interval posets~\cite{pons19intervalposets}.  The specialization to the setting of Dyck paths can be deduced from \cite{pons19intervalposets}*{Theorem~23}.

We then prove Theorem~\ref{thm:alpha_tamari_isomorphism} in
Section~\ref{sec:alpha_tamari_lattice_isomorphism} by describing an explicit
conversion from parabolic BW-codes to $\nu_{\alpha}$-bracket vectors.
Finally, in Section~\ref{sec:stack}, we give a combinatorial interpretation of the map 
$\Theta$ mentioned after Theorem~\ref{thm:alpha_tamari_isomorphism} in terms of a certain stack-processing procedure, and relate the bijection
between parabolic BW-codes and $\nu_{\alpha}$-bracket vectors to $\Theta$, thus
solving \cite{ceballos20the}*{Open Problem~2.23}.

\section{$\alpha$-permutations and the $\alpha$-Tamari lattice}
\label{sec:basics}

Throughout this article, we fix an integer $n>0$ and define
$[n]\defs\{1,2,\ldots,n\}$.

\subsection{$\alpha$-permutations}
\label{sec:alpha_permutations}

Let $\alpha=(\alpha_{1},\alpha_{2},\ldots,\alpha_{r})$ be a composition of
$n$. For $a\in[r]$, we define 
\begin{displaymath}
	s_{a}\defs\alpha_{1}+\alpha_{2}+\cdots+\alpha_{a},
\end{displaymath}
and we set $s_{0}\defs 0$.  The set $\{s_{a-1}{+}1,s_{a-1}{+}2,\ldots,s_{a}\}$
is the $a\th$ \defn{$\alpha$-region}.

The \defn{indicator map} $\indicator_{\alpha}\colon[n]\to[r]$ is defined by
$\indicator_{\alpha}(i)=a$ if, and only if $s_{a-1}<i\leq s_{a}$.  In other
words, $\indicator_{\alpha}(i)$ is the index of the $\alpha$-region containing
$i$.  When no confusion will arise, we will drop the subscript $\alpha$.  For
three indices $i<j<k$ with $\indicator(i)<\indicator(k)$, we say that $j$ is
in an $\alpha$-region \defn{strictly between} $i$ and $k$ if
$\indicator(i)<\indicator(j)<\indicator(k)$.

Let $\Symmetric_{n}$ denote the \defn{symmetric group} of degree $n$.  We
consider the subset of \defn{$\alpha$-permutations}, defined by
\begin{displaymath}
	\Symmetric_{\alpha} \defs \{w\in\Symmetric_{n}\mid\text{if}\;\indicator(i)=\indicator(i+1),\;\text{then}\;w(i)<w(i+1)\}.
\end{displaymath}
Clearly, if $\alpha=(1,1,\ldots,1)$, then $\Symmetric_{\alpha}=\Symmetric_{n}$.

\begin{remark}
If we consider the subgroup
$G\defs\Symmetric_{\lvert\alpha_{1}\rvert}\times\Symmetric_{\lvert\alpha_{2}\rvert}\times\cdots\times\Symmetric_{\lvert\alpha_{r}\rvert}$
of $\Symmetric_{n}$, then we may identify $\Symmetric_{\alpha}$ with the set
of minimal-length representatives of the left cosets in $\Symmetric_{n}/G$.
\end{remark}

An $\alpha$-permutation $w\in\Symmetric_{\alpha}$ has an
\defn{$(\alpha,231)$-pattern} if there are three indices $i<j<k$---each in
different $\alpha$-regions---such that $w_{i}<w_{j}$ and $w_{i}=w_{k}+1$.  If
$w$ does not have an $(\alpha,231)$-pattern, then $w$ is
\defn{$(\alpha,231)$-avoiding}. Let $\Symmetric_{\alpha}(231)$ denote the set
of \defn{$(\alpha,231)$-avoiding permutations}.

\begin{remark}
In the case $\alpha=(1,1,\ldots,1)$, the $(\alpha,231)$-avoiding permutations are exactly the \emph{classical} $231$-avoiding permutations:
one can either ask $w_i=w_k+1$ or not, since if $w$ has
any $231$-pattern, then one can find one with the extra condition $w_i=w_k+1$.
In the general case, these notions differ since $i$ could belong to the same
$\alpha$-region as $j$.
For example, $3\;24\;1$ belongs to
$\Symmetric_{(1,2,1)}(231)$ whereas it has a classical $231$-pattern 
spread out over different $\alpha$-regions.
\end{remark}

\subsection{The weak order}
\label{sec:weak_order}

For $w\in\Symmetric_{\alpha}$, we define its \defn{(right) inversion set} by
\begin{displaymath}
	\Inv(w) \defs \bigl\{(i,j)\mid i<j\;\text{and}\;w_{i}>w_{j}\bigr\}.
\end{displaymath}
This enables us to define a partial order---the \defn{(left) weak order}---on
$\Symmetric_{\alpha}$ by setting
\begin{displaymath}
	u\leq_{L}v\quad\text{if, and only if}\quad\Inv(u)\subseteq\Inv(v).
\end{displaymath}
Two permutations $u,v\in\Symmetric_{\alpha}$ form a \defn{cover
relation}---denoted by $u\lessdot_{L}v$---if $u<_{L}v$ and there is no
$w\in\Symmetric_{\alpha}$ with $u<_{L}w<_{L}v$. One easily checks that
$u\lessdot_{L}v$ if, and only if there are two indices $i<j$ in different
$\alpha$-regions, such that $u_{i}=u_{j}-1$, and 
\begin{displaymath}
	v_{k} = \begin{cases}u_{j}, & \text{if}\;k=i,\\ u_{i}, & \text{if}\;k=j,\\ u_{k}, & \text{otherwise}.\end{cases}
\end{displaymath}

The partially ordered set $(\Symmetric_{n},\weakorder)$ is a lattice by
\cite{yanagimoto69partial}*{Theorem~2.1}; see also \cite{guilbaud71analyse}.
For an arbitrary composition $\alpha$ of $n$, it follows from
\cite{bjorner88generalized}*{Theorem~4.1} that
$(\Symmetric_{\alpha},\weakorder)$ is an interval of
$(\Symmetric_{n},\weakorder)$, and thus also a lattice.

The partially ordered set
$\Tamari_{\alpha}\defs\bigl(\Symmetric_{\alpha}(231),\weakorder\bigr)$ is the
\defn{$\alpha$-Tamari lattice}. This name is justified by the following
result.

\begin{theorem}[\cite{muehle19tamari}*{Theorem~1}]\label{thm:alpha_tamari_lattice}
$\Tamari_{\alpha}$ is a lattice for every $n>0$ and every integer composition
$\alpha$ of $n$.
\end{theorem}

\section{A generalized Lehmer code for $\Symmetric_\alpha$}
\label{sec:alpha_codes}

\subsection{Encoding $\alpha$-permutations}
\label{sec:alpha_encoding}

We consider the following set of integer tuples.  

\begin{definition} \label{def:all_codes}
Let $\Codes_{\alpha}$ denote the set of all integer tuples
$(c_{1},c_{2},\ldots,c_{n})$ with the following properties:
	\begin{description}
		\item[C1\label{it:code1}] $0\leq c_{i}\leq n-s_{\indicator(i)}$ for all $i\in[n]$;
		\item[C2\label{it:code2}] $c_{i}\leq c_{i+1}$ for all $i\in[n-1]$ such that $\indicator(i)=\indicator(i+1)$;
		\item[C3\label{it:code3}] $c_{s_{a}}\leq
c_{i}-s_{a}+s_{\indicator(i)}$ for all $i\in[s_{r-2}]$ and all
$a\in\bigl\{\indicator(i){+}1,\indicator(i){+}2,\ldots,r{-}1\bigr\}$ such that
$c_{i}\geq s_{a}-s_{\indicator(i)}$.
	\end{description}
\end{definition}

The set $\Codes_{(1,1,\ldots,1)}$ is precisely the set of integer tuples
defined in \cite{bjorner97shellable}*{Definition~9.1}.

\begin{remark}
The statement of \eqref{it:code3} is directly true if $a=r$ and trivial
if $i>s_{r-2}$, hence the restriction to $i\in[s_{r-2}]$ and $a<r$.

Indeed, by \eqref{it:code1}, $c_{n}=0$ so that the implication required by
\eqref{it:code3} is trivially satisfied when $a=r$. If $i>s_{r-2}$, then
$\indicator(i)\geq r-1$, so that the only case one could consider is again
$a=r$.
\end{remark}

For example, with $n=3$ and $\alpha=(2,1)$, all conditions boil down to
$0\leq c_1\leq c_2\leq 1$ and $c_3=0$, hence three solutions. With $n=3$ and
$\alpha=(1,2)$, one gets $0\leq c_1\leq 2$ and $0\leq c_2\leq c_3\leq 0$,
again providing three solutions. One can check that they are indeed the codes
obtained in Table~\ref{tab:induction_base} (right column).

To see all conditions of the definition play a role, one has to consider
compositions of at least three parts and at least one greater than one.
For example, if
$\alpha=(1,2,1)$, one gets the following set of relations:
$0\leq c_1\leq 3$, $0\leq c_2\leq c_3\leq 1$, $c_4=0$, and the extra condition
coming from \eqref{it:code3}: $c_1\geq2 \Rightarrow c_3\leq c_1-2$.
In practice, we have twelve tuples satisfying all conditions except the last
one and this last condition gets rid of $(2,0,1,0)$ and $(2,1,1,0)$, hence
providing a total of ten solutions.
One can check that these solutions are exactly the codes obtained in
Figure~\ref{fig:121_tamari_lattice} (bottom elements in each cell
of the drawing).

\smallskip
Given two tuples $\mathbf{a}=(a_{1},a_{2},\ldots,a_{n})$ and
$\mathbf{b}=(b_{1},b_{2},\ldots,b_{n})$ we write
$\mathbf{a}\comporder\mathbf{b}$ if $a_{i}\leq b_{i}$ for all $i\in[n]$. We
claim in Theorem~\ref{thm:alpha_tamari_lattice_codes} that the poset
${(\Codes_{\alpha},\comporder)}$ is isomorphic to $\Tamari_{\alpha}$.
For example, one can check, again on Figure~\ref{fig:121_tamari_lattice} that
the bottom elements are indeed (partially) ordered by the componentwise order
on their tuples.

As a first step towards proving Theorem~\ref{thm:alpha_tamari_lattice_codes},
we associate an integer tuple with each $w\in\Symmetric_{\alpha}$.

\begin{definition}\label{def:perm_to_code}
For $w\in\Symmetric_{\alpha}$, we define its \defn{$\alpha$-code}
by
	\begin{displaymath}
		\code_{\alpha}(w)\defs (c_{1},c_{2},\ldots,c_{n}),
	\end{displaymath}
	where 
	\begin{displaymath}
		c_{i} \defs \max\bigl\{k\mid w_{i}>w_{s_{\indicator(i)}+1},w_{i}>w_{s_{\indicator(i)}+2},\ldots,w_{i}>w_{s_{\indicator(i)}+k}\bigr\}.
	\end{displaymath}
\end{definition}

In other words, $c_{i}$ counts the number of consecutive entries in the
one-line notation of $w$ that are smaller than $w_{i}$, starting from the
first entry in the $\alpha$-region immediately after that of $i$.
For $\alpha=(1,1,\ldots,1)$, Definition~\ref{def:perm_to_code} agrees with
\cite{bjorner97shellable}*{Definition~9.9}.

If $\code_{\alpha}(w)=(c_{1},c_{2},\ldots,c_{n})$, then we say that $w_{i}$
\defn{sees} $w_{k}$ if $0<k-s_{\indicator(i)}\leq c_{i}$. Clearly, if $w_{i}$
sees $w_{k}$, then $(i,k)\in\Inv(w)$, and $w_{i}$ sees exactly $c_{i}$
elements for each index $i$.

In terms of patterns, $c_i$ is the number of $21$-patterns where the $2$ is at
position~$i$ that are not $231$-patterns.
Examples of codes of $\alpha$-permutations are shown in
Table~\ref{tab:induction_base} and Figure~\ref{fig:121_tamari_lattice}.

\subsection{Properties of the encoding}
\label{sec:alpha_code_properties}

\begin{lemma}\label{lem:alpha_code_is_code}
For $w\in\Symmetric_{\alpha}$ it holds that
$\code_{\alpha}(w)\in\mathcal{C}_{\alpha}$.
\end{lemma}

\begin{proof}
Let $w\in\Symmetric_{\alpha}$ and
$\code_{\alpha}(w)=(c_{1},c_{2},\ldots,c_{n})$.  Let $i\in[n]$.  The maximal
number of inversions of the form $(i,k)$ is $n-s_{\indicator(i)}$, because
$w_{i}<w_{k}$ for all $k\in\{i{+}1,i{+}2,\ldots,s_{\indicator(i)}\}$.
Hence, $c_{i}\leq n-s_{\indicator(i)}$, which establishes \eqref{it:code1}.
If $\indicator(i)={\indicator(i+1)}$, then $w_{i}<w_{i+1}$ by construction,
and thus $c_{i}\leq c_{i+1}$. This establishes \eqref{it:code2}.
	
Now let $k\in\bigl\{\indicator(i){+}1,\indicator(i){+}2,\ldots,r\bigr\}$ be
such that $c_{i}\geq s_{k}-s_{\indicator(i)}$. In particular, $w_{i}$ sees
$w_{s_{k}}$, meaning that $w_{i}>w_{s_{k}}$.
Since $w_{s_k}$ is the rightmost hence largest element of its
region, $w_{i}$ also sees any $w_{j}$ which is seen by $w_{s_{k}}$.
This implies that $c_{i}\geq c_{s_{k}}+s_{k}-s_{\indicator(i)}$, which is
\eqref{it:code3}.
\end{proof}

Theorem~1.1 in \cite{muehle19tamari} establishes that the $\alpha$-Tamari
lattice arises as a quotient lattice of the weak order on
$\Symmetric_{\alpha}$. This is established by proving that for every
$w\in\Symmetric_{\alpha}$ there exists a unique maximal
$(\alpha,231)$-avoiding permutation below $w$ in the weak order.
The next lemma records this fact.

\begin{lemma}[\cite{muehle19tamari}*{Lemma~3.8}]
\label{lem:alpha_projection}
For $w\in\Symmetric_{\alpha}$, the set $\{w'\in\Symmetric_{\alpha}(231)\mid
w'\weakorder w\}$ has a greatest element denoted by $\proj_{\alpha}(w)$.
\end{lemma}

We may thus regard $\proj_{\alpha}$ as a map from $\Symmetric_{\alpha}$ to
$\Symmetric_{\alpha}(231)$.  The next lemma characterizes the preimages of
this map.

\begin{lemma}[\cite{muehle19tamari}*{Lemma~3.16}]\label{lem:alpha_congruence}
Let $u,v\in\Symmetric_{\alpha}$ with $u\lessdot_{L}v$.  The following are
equivalent.
\begin{enumerate}[\rm (i)]
\item There are indices $i<j<k$, each in different $\alpha$-regions, such that
$v_{k}<v_{i}<v_{j}$, $v_{i}=v_{k}+1$ and
$\Inv(v)\setminus\Inv(u)=\bigl\{(i,k)\bigr\}$,
\item $\proj_{\alpha}(u)=\proj_{\alpha}(v)$.
\end{enumerate}
\end{lemma}

We now prove that $\code_{\alpha}$ is an order-preserving map from
$(\Symmetric_{\alpha},\weakorder)$ to ${(\Codes_{\alpha},\comporder)}$.

\begin{lemma}\label{lem:code_steps}
Let $u,v\in\Symmetric_{\alpha}$ with $u\lessdot_{L}v$. Then
$\code_{\alpha}(u)\comporder\code_{\alpha}(v)$, and these tuples differ by at
most one element. Moreover, $\code_{\alpha}(u)=\code_{\alpha}(v)$ if and
only if $\proj_{\alpha}(u)=\proj_{\alpha}(v)$.
\end{lemma}

\begin{proof}
Let $u\lessdot_{L}v$ and $\code_{\alpha}(u)=(a_{1},a_{2},\ldots,a_{n})$ and
$\code_{\alpha}(v)=(b_{1},b_{2},\ldots,b_{n})$.

By assumption, $\Inv(v)\setminus\Inv(u)=\bigl\{(i,k)\bigr\}$ for some indices
$i<k$ in different $\alpha$-regions such that $v_{i}=v_{k}+1$. It follows
that any entry which sees $v_{k}$ must be greater than $v_{i}$, and any entry
which does not see $v_{i}$ must be smaller than $v_{k}$.  Thus, $a_{j}=b_{j}$
for all $j\neq i$.

By construction, $u_{i}=v_{k}$ and $u_{k}=v_{i}$. Since $u_{i}<u_{k}$, we
conclude that $u_{i}$ does not see $u_{k}$.

If $v_{i}$ sees $v_{k}$, then $a_{i}<b_{i}$. This is the case precisely when
every $j$ in $\alpha$-regions strictly between $i$ and $k$ satisfies
$v_{j}<v_{i}$, which by Lemma~\ref{lem:alpha_congruence} means that
$\proj_{\alpha}(u)\neq\proj_{\alpha}(v)$.
	
If $v_{i}$ does not see $v_{k}$, then there exists an index $j$ in an
$\alpha$-region strictly between $i$ and $k$ such that $v_{i}<v_{j}$, which by
Lemma~\ref{lem:alpha_congruence} is equivalent to
$\proj_{\alpha}(u)=\proj_{\alpha}(v)$.  If we choose $j$ as small as possible
with this property, then any $j'<j$ in an $\alpha$-region strictly between $i$
and $k$ satisfies $v_{i}>v_{j'}$, and thus $u_{i}=v_{k}>v_{j'}=u_{j'}$, which
entails $a_{i}=b_{i}$.
\end{proof}

\begin{corollary}\label{cor:code_comparison_necessary}
If $u\weakorder v$, then $\code_{\alpha}(u)\comporder\code_{\alpha}(v)$.
\end{corollary}

\begin{proof}
This follows from repeated application of Lemma~\ref{lem:code_steps}.
\end{proof}

\begin{lemma}\label{lem:code_comparison_sufficient}
Let $u,v\in\Symmetric_{\alpha}(231)$.  If
$\code_{\alpha}(u)\comporder\code_{\alpha}(v)$, then $u\weakorder v$.
\end{lemma}

\begin{proof}
Let $\code_{\alpha}(u)=(a_{1},a_{2},\ldots,a_{n})$ and
$\code_{\alpha}(v)=(b_{1},b_{2},\ldots,b_{n})$ such that
$\code_{\alpha}(u)\comporder\code_{\alpha}(v)$.
	
Assume that there exists $(i,k)\in\Inv(u)\setminus\Inv(v)$, and among all
these inversions choose $(i,k)$ such that $u_{i}-u_{k}$ is minimal.
Since $(i,k)$ is not an inversion of $v$, we have $v_{i}<v_{k}$, so that
$v_{i}$ does not see $v_{k}$. Since $a_{i}\leq b_{i}$ it follows that $u_{i}$
does not see $u_{k}$ either. Since $u_{i}>u_{k}$,
$\indicator(i)<\indicator(k)$ and there exists a smallest index $j$ with
$\indicator(i)<\indicator(j)<\indicator(k)$ and $u_{i}<u_{j}$. Since
$u\in\Symmetric_{\alpha}(231)$, we have that $u_{i}>u_{k}+1$.

Now, there cannot be any element between $u_k+1$ and $u_i-1$ in the same
$\alpha$-region as $u_j$. Indeed, if this was the case, since $u_j>u_i$,
$u_{j-1}$ would be such an element. But, since it is seen by $u_i$ by
minimality of $j$, and since $b_{i}\geq a_i$, the value $v_{j-1}$ would be
seen by $v_i$, so that $v_k>v_i>v_{j-1}$. In that case, $(j-1,k)$ would be an
inversion of $u$, not an inversion of $v$ and would violate the minimality of
$(i,k)$ among such elements as defined earlier.
So all elements between $u_k$ and $u_i$ belong to $\alpha$-regions different
from the $\alpha$-region containing $u_j$.
Thus, among those, there is a smallest one $u_{\ell}$ (which is not $u_k$ but
can be $u_i$) that is on the left of $u_j$. This element
belongs to an $(\alpha,231)$-pattern in $u$: $(\ell,j,\ell')$, where $\ell'$
is the position of $u_{\ell}-1$ in $u$, which is a contradiction.

Therefore, our assumption must have been wrong, and it follows
$\Inv(u)\subseteq\Inv(v)$, thus $u\weakorder v$ by definition.
\end{proof}

Note that we never used in the previous proof that $v$ is
$(\alpha,231)$-avoiding. This makes sense thanks to
Lemma~\ref{lem:alpha_projection}: its property has no equivalent going
upwards so $u$ and $v$ do not play symmetrical roles.

\subsection{Decoding $\alpha$-codes}
\label{sec:alpha_decoding}

We proceed to prove that $\code_{\alpha}$ is a bijection from
$\Symmetric_{\alpha}(231)$ to $\Codes_{\alpha}$.
	
\begin{lemma}\label{lem:code_left_zero}
If $w\in\Symmetric_{\alpha}(231)$, then the leftmost $0$ in
$\code_{\alpha}(w)$ corresponds to the position of the $1$ in the one-line
notation of $w$.
\end{lemma}

\begin{proof}
Let $\code_{\alpha}(w)=(c_{1},c_{2},\ldots,c_{n})$, and let $j_{o}\in[n]$ be
such that $w_{j_{o}}=1$. Moreover, if $j=\min\{i\mid c_{i}=0\}$, then $j\leq
j_{o}$, since $c_{j_{o}}=0$. Let $w_{j}=a$. Since the entries in
an $\alpha$-region are ordered increasingly, $a$ is the smallest element in
its $\alpha$-region.

Now, define $m=\min(w_1,w_2,\ldots,w_j)$. Then, if $m\not=a$, since $m$ is
strictly to the left of $a$ in $w$, it cannot be $1$ either, so that we have
a $231$ pattern with the values $m$, $a$, and $m-1$ necessarily in that order
in $w$ and in different regions. Otherwise $m=a$. If $m\not=1$, the region
of $a$ cannot be the rightmost region of $w$ since $a-1$ did not appear in
this prefix of $w$.
Let $b$ be the smallest element in the
$\bigl(\indicator(j)+1\bigr)$-st $\alpha$-region, and let
$k=s_{\indicator(j)}+1$, \ie $w_{k}=b$. Since $c_{j}=0$, we have $a<b$.
We then have a $231$ pattern with the values $a$, $b$, and $a-1$.

%
%
So $a=1$ and $j=j_{o}$.
\end{proof}

\begin{table}
	\centering
	\begin{tabular}{c||c|c}
		$\alpha$ & $w$ & $\code_{\alpha}(w)$\\
		\hline\hline
		$(1)$ & \raisebox{-.1cm}{\AlphaPerm{1}{c1}{1}{.3}{.25}} & $(0)$\\
		\hline
		$(2)$ & \raisebox{-.1cm}{\AlphaPerm{2}{c1}{1,2}{.3}{.25}} & $(0,0)$\\
		$(1,1)$ & \raisebox{-.1cm}{\AlphaPerm{1,1}{c1,c2}{1,2}{.3}{.25}} & $(0,0)$\\
		& \raisebox{-.1cm}{\AlphaPerm{1,1}{c1,c2}{2,1}{.3}{.25}} & $(1,0)$\\
		\hline
		$(3)$ & \raisebox{-.1cm}{\AlphaPerm{3}{c1}{1,2,3}{.3}{.25}} & $(0,0,0)$\\
		$(2,1)$ & \raisebox{-.1cm}{\AlphaPerm{2,1}{c1,c2}{1,2,3}{.3}{.25}} & $(0,0,0)$\\
		& \raisebox{-.1cm}{\AlphaPerm{2,1}{c1,c2}{1,3,2}{.3}{.25}} & $(0,1,0)$\\
		& \raisebox{-.1cm}{\AlphaPerm{2,1}{c1,c2}{2,3,1}{.3}{.25}} & $(1,1,0)$\\
		$(1,2)$ & \raisebox{-.1cm}{\AlphaPerm{1,2}{c1,c2}{1,2,3}{.3}{.25}} & $(0,0,0)$\\
		& \raisebox{-.1cm}{\AlphaPerm{1,2}{c1,c2}{2,1,3}{.3}{.25}} & $(1,0,0)$\\
		& \raisebox{-.1cm}{\AlphaPerm{1,2}{c1,c2}{3,1,2}{.3}{.25}} & $(2,0,0)$\\
		$(1,1,1)$ & \raisebox{-.1cm}{\AlphaPerm{1,1,1}{c1,c2,c3}{1,2,3}{.3}{.25}} & $(0,0,0)$\\
		& \raisebox{-.1cm}{\AlphaPerm{1,1,1}{c1,c2,c3}{1,3,2}{.3}{.25}} & {\color{red!80!black}$(0,1,0)$}\\
		& \raisebox{-.1cm}{\AlphaPerm{1,1,1}{c1,c2,c3}{2,1,3}{.3}{.25}} & $(1,0,0)$\\
		& \raisebox{-.1cm}{\AlphaPerm{1,1,1}{c1,c2,c3}{2,3,1}{.3}{.25}} & {\color{red!80!black}$(0,1,0)$}\\
		& \raisebox{-.1cm}{\AlphaPerm{1,1,1}{c1,c2,c3}{3,1,2}{.3}{.25}} & $(2,0,0)$\\
		& \raisebox{-.1cm}{\AlphaPerm{1,1,1}{c1,c2,c3}{3,2,1}{.3}{.25}} & $(2,1,0)$\\
	\end{tabular}
	\caption{The $\alpha$-permutations for any composition $\alpha$ of $n\leq 3$ together with their corresponding $\alpha$-codes.}
	\label{tab:induction_base}
\end{table}

\begin{proposition}\label{prop:code_bijection}
For $\mathbf{c}\in\Codes_{\alpha}$ there exists a unique
$w\in\Symmetric_{\alpha}(231)$ such that $\code_{\alpha}(w)=\mathbf{c}$.
\end{proposition}

\begin{proof}
We proceed by induction on $n$. For $n\leq 3$, the claim can be checked
directly (see Table~\ref{tab:induction_base}), which establishes the induction
base. Assume that the claim holds for all compositions of $n'<n$.
	
Let $\mathbf{c}=(c_{1},c_{2},\ldots,c_{n})\in\Codes_{\alpha}$.  By definition,
$c_{n}=0$, which enables us to define $j_{o}=\min\bigl\{j\in[n]\mid
c_{j}=0\bigr\}$.  By \eqref{it:code2}, $j_{o}=s_{a-1}+1$ for some $a\in[r]$,
meaning that $j_{o}$ is the first element in the $a\th$ $\alpha$-region.
	
Let $\alpha'=(\alpha'_{1},\alpha'_{2},\ldots,\alpha'_{r'})$ be the unique
composition of $n-1$ which is obtained by subtracting $1$ from $\alpha_{a}$.
(If $\alpha_{a}=1$, then we simply remove this part.)  We define
$s'_{a}=\alpha'_{1}+\alpha'_{2}+\cdots+\alpha'_{a}$, and we obtain

\begin{displaymath}
s'_{b} = \begin{cases}s_{b}, & \text{if}\;b<a,\\ s_{b}-1, & \text{if}\;b\geq a.\end{cases}
\end{displaymath}
We define $\mathbf{c}'=(c'_{1},c'_{2},\ldots,c'_{n-1})$ by setting
\begin{displaymath}
c'_{i} = 
\begin{cases}
c_{i}, & i<j_{o}\;\text{and}\;c_{i}<s_{a-1}-s_{\indicator_{\alpha}(i)},\\
c_{i}-1, & i<j_{o}\;\text{and}\;c_{i}\geq s_{a-1}-s_{\indicator_{\alpha}(i)},\\
c_{i+1}, & i\geq j_{o}.
\end{cases}
\end{displaymath}
It is straightforward to check that $\mathbf{c}'\in\Codes_{\alpha'}$.  By
induction hypothesis, there exists a unique $w'\in\Symmetric_{\alpha'}(231)$
with $\code_{\alpha'}(w')=\mathbf{c}'$.
	
We now ``inject'' $1$ into $w'$ to construct a permutation
$w\in\Symmetric_{n}$ via
	\begin{displaymath}
		w_{i} =
		\begin{cases}
			w'_{i}+1, & \text{if}\;i<j_{o},\\
			1, & \text{if}\;i=j_{o},\\
			w'_{i-1}+1, & \text{if}\;i>j_{o}.
		\end{cases}
	\end{displaymath}
Since $j_{o}$ is the first element in the $a\th$ $\alpha$-region, it follows
that $w\in\Symmetric_{\alpha}$.  Assume that $w$ has an $(\alpha,231)$-pattern
$(i,j,k)$.  Since $w'\in\Symmetric_{\alpha'}(231)$, it must be that $k=j_{o}$,
and $w_{i}=2$.  By construction, $w'_{i}=1$, implying that $c'_{i}=0$.  Since
$i<j_{o}$, it follows that $c_{i}=0$, contradicting the choice of $j_{o}$.
Thus, $w\in\Symmetric_{\alpha}(231)$.  By construction, it follows that $w$ is
the only $(\alpha,231)$-avoiding permutation with
$\code_{\alpha}(w)=\mathbf{c}$.
\end{proof}

\begin{figure}
	\centering
	\begin{tikzpicture}\small
		\def\x{1};
		\def\y{.8};
		\draw(1*\x,8*\y) node{$({\color{red!80!black}\mathbf{2}},{\color{red!80!black}\mathbf{6}},0,1,3,1,1,0)$};
			\draw(.57*\x,7.5*\y) node{\tiny $\uparrow$};
		\draw(1*\x,7*\y) node{$(1,{\color{red!80!black}\mathbf{5}},\ndash,1,{\color{red!80!black}\mathbf{3}},{\color{red!80!black}\mathbf{1}},{\color{red!80!black}\mathbf{1}},0)$};
			\draw(2.03*\x,6.5*\y) node{\tiny $\uparrow$};
		\draw(1*\x,6*\y) node{$(1,{\color{red!80!black}\mathbf{4}},\ndash,{\color{red!80!black}\mathbf{1}},{\color{red!80!black}\mathbf{2}},0,0,\ndash)$};
			\draw(1.44*\x,5.5*\y) node{\tiny $\uparrow$};
		\draw(1*\x,5*\y) node{$({\color{red!80!black}\mathbf{1}},{\color{red!80!black}\mathbf{3}},\ndash,0,1,\ndash,0,\ndash)$};
			\draw(.86*\x,4.5*\y) node{\tiny $\uparrow$};
		\draw(1*\x,4*\y) node{$(0,2,\ndash,\ndash,1,\ndash,0,\ndash)$};
			\draw(-.02*\x,3.5*\y) node{\tiny $\uparrow$};
		\draw(1*\x,3*\y) node{$(\ndash,{\color{red!80!black}\mathbf{2}},\ndash,\ndash,{\color{red!80!black}\mathbf{1}},\ndash,0,\ndash)$};
			\draw(1.73*\x,2.5*\y) node{\tiny $\uparrow$};
		\draw(1*\x,2*\y) node{$(\ndash,{\color{red!80!black}\mathbf{1}},\ndash,\ndash,0,\ndash,\ndash,\ndash)$};
			\draw(1.14*\x,1.5*\y) node{\tiny $\uparrow$};
		\draw(1*\x,1*\y) node{$(\ndash,0,\ndash,\ndash,\ndash,\ndash,\ndash,\ndash)$};
			\draw(.27*\x,.5*\y) node{\tiny $\uparrow$};
		\draw(3*\x,8*\y) node{$\rightarrow$};
		\draw(3*\x,7*\y) node{$\rightarrow$};
		\draw(3*\x,6*\y) node{$\rightarrow$};
		\draw(3*\x,5*\y) node{$\rightarrow$};
		\draw(3*\x,4*\y) node{$\rightarrow$};
		\draw(3*\x,3*\y) node{$\rightarrow$};
		\draw(3*\x,2*\y) node{$\rightarrow$};
		\draw(3*\x,1*\y) node{$\rightarrow$};
		\draw(5*\x,8.03*\y) node{\AlphaPerm{2,3,2,1}{c1,c2,c3,c4}{,,1,,,,,}{.3}{.25}};
		\draw(5*\x,7.03*\y) node{\AlphaPerm{2,3,2,1}{c1,c2,c3,c4}{,,1,,,,,2}{.3}{.25}};
		\draw(5*\x,6.03*\y) node{\AlphaPerm{2,3,2,1}{c1,c2,c3,c4}{,,1,,,3,,2}{.3}{.25}};
		\draw(5*\x,5.03*\y) node{\AlphaPerm{2,3,2,1}{c1,c2,c3,c4}{,,1,4,,3,,2}{.3}{.25}};
		\draw(5*\x,4.03*\y) node{\AlphaPerm{2,3,2,1}{c1,c2,c3,c4}{5,,1,4,,3,,2}{.3}{.25}};
		\draw(5*\x,3.03*\y) node{\AlphaPerm{2,3,2,1}{c1,c2,c3,c4}{5,,1,4,,3,6,2}{.3}{.25}};
		\draw(5*\x,2.03*\y) node{\AlphaPerm{2,3,2,1}{c1,c2,c3,c4}{5,,1,4,7,3,6,2}{.3}{.25}};
		\draw(5*\x,1.03*\y) node{\AlphaPerm{2,3,2,1}{c1,c2,c3,c4}{5,8,1,4,7,3,6,2}{.3}{.25}};
		\begin{pgfonlayer}{background}
			\foreach \a in {0,1,...,7}{
				\draw[thick,c1](0*\x,\a.8*\y) -- (0*\x,\a.7*\y) -- (.25*\x,\a.7*\y) -- (.25*\x,\a.8*\y);
				\draw[thick,c2](.55*\x,\a.8*\y) -- (.55*\x,\a.7*\y) -- (1.15*\x,\a.7*\y) -- (1.15*\x,\a.8*\y);
				\draw[thick,c3](1.45*\x,\a.8*\y) -- (1.45*\x,\a.7*\y) -- (1.7*\x,\a.7*\y) -- (1.7*\x,\a.8*\y);
				\draw[thick,c4](2*\x,\a.8*\y) -- (2*\x,\a.7*\y);
			}
		\end{pgfonlayer}
	\end{tikzpicture}
	\caption{Decoding the $(2,3,2,1)$-code $(2,6,0,1,3,1,1,0)$.  The arrows indicate the left-most zero in each step; the red digits indicate the positions that see the left-most zero.}
	\label{fig:alpha_decoding}
\end{figure}
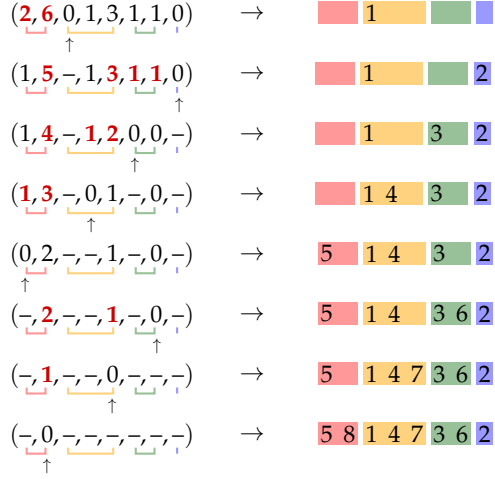

Figure~\ref{fig:alpha_decoding} illustrates the procedure described in the
proof of Proposition~\ref{prop:code_bijection}.  We may now conclude the proof
of our first main theorem.

\begin{proof}[Proof of Theorem~\ref{thm:alpha_tamari_lattice_codes}]
Proposition~\ref{prop:code_bijection} establishes that
$\Symmetric_{\alpha}(231)$ and $\Codes_{\alpha}$ are in bijection, and
Corollary~\ref{cor:code_comparison_necessary} and
Lemma~\ref{lem:code_comparison_sufficient} establish that for
$u,v\in\Symmetric_{\alpha}(231)$ we have $u\weakorder v$ if, and only if
$\code_{\alpha}(u)\comporder\code_{\alpha}(v)$.  This finishes the proof.
\end{proof}

In fact, the preimages of the maps
$\code_{\alpha}\colon\Symmetric_{\alpha}\to\Codes_{\alpha}$ and
$\proj_{\alpha}\colon\Symmetric_{\alpha}\to\Symmetric_{\alpha}(231)$
coincide.

\begin{lemma}\label{lem:code_congruence}
For $u,v\in\Symmetric_{\alpha}$ we have $\code_{\alpha}(u)=\code_{\alpha}(v)$
if, and only if $\proj_{\alpha}(u)=\proj_{\alpha}(v)$.
\end{lemma}

\begin{proof}
Let $u,v\in\Symmetric_{\alpha}$.  Let
$\code_{\alpha}(u)=(a_{1},a_{2},\ldots,a_{n})$ and
$\code_{\alpha}(v)=(b_{1},b_{2},\ldots,b_{n})$.  
	
If $u\weakorder v$, then the desired equivalence follows from repeated
application of Lemma~\ref{lem:code_steps}.  
	
Otherwise, $u$ and $v$ are incomparable.  By
\cite{bjorner88generalized}*{Theorem~4.1}, $(\Symmetric_{\alpha},\weakorder)$
is a lattice and thus the meet $w=u\wedge_{L}v$ exists and satisfies
$w\weakorder u$ and $w\weakorder v$.  If
$\proj_{\alpha}(u)=\proj_{\alpha}(v)$, then
$\proj_{\alpha}(u)=\proj_{\alpha}(w)$, by Lemma~\ref{lem:alpha_projection}.
It follows that $\code_{\alpha}(u)=\code_{\alpha}(w)=\code_{\alpha}(v)$ by
Lemma~\ref{lem:code_steps}.
	
Conversely, let $\code_{\alpha}(u)=\code_{\alpha}(v)$.
Lemma~\ref{lem:alpha_projection} implies $\proj_{\alpha}(u)\weakorder u$ and
$\proj_{\alpha}(v)\weakorder v$.  In view of the previous reasoning we find
$\code_{\alpha}\bigl(\proj_{\alpha}(u)\bigr)=\code_{\alpha}(u)=\code_{\alpha}(v)=\code_{\alpha}\bigl(\proj_{\alpha}(v)\bigr)$.
Proposition~\ref{prop:code_bijection} thus implies
$\proj_{\alpha}(u)=\proj_{\alpha}(v)$.
\end{proof}

Figure~\ref{fig:121_tamari_lattice} shows the weak order on
$\Symmetric_{(1,2,1)}$ with the preimages of the map $\proj_{(1,2,1)}$
indicated; the bottom elements per highlighted region are exactly the
$\bigl((1,2,1),231\bigr)$-avoiding permutations.  The elements are labeled by
their corresponding $(1,2,1)$-codes, too.

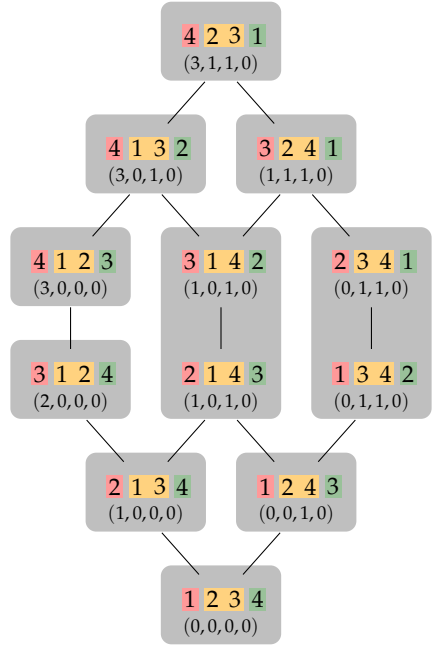
\begin{figure}
	\centering
	\begin{tikzpicture}\small
		\def\x{2};
		\def\y{1.5};
		\draw(2*\x,1*\y) node(n1){\AlphaPerm{1,2,1}{c1,c2,c3}{1,2,3,4}{.3}{.25}};
			\draw(2*\x,.75*\y) node(nl1){\scriptsize $(0,0,0,0)$};
		\draw(1.5*\x,2*\y) node(n2){\AlphaPerm{1,2,1}{c1,c2,c3}{2,1,3,4}{.3}{.25}};
			\draw(1.5*\x,1.75*\y) node(nl2){\scriptsize $(1,0,0,0)$};
		\draw(2.5*\x,2*\y) node(n3){\AlphaPerm{1,2,1}{c1,c2,c3}{1,2,4,3}{.3}{.25}};
			\draw(2.5*\x,1.75*\y) node(nl3){\scriptsize $(0,0,1,0)$};
		\draw(1*\x,3*\y) node(n4){\AlphaPerm{1,2,1}{c1,c2,c3}{3,1,2,4}{.3}{.25}};
			\draw(1*\x,2.75*\y) node(nl4){\scriptsize $(2,0,0,0)$};
		\draw(2*\x,3*\y) node(n5){\AlphaPerm{1,2,1}{c1,c2,c3}{2,1,4,3}{.3}{.25}};
			\draw(2*\x,2.75*\y) node(nl5){\scriptsize $(1,0,1,0)$};
		\draw(3*\x,3*\y) node(n6){\AlphaPerm{1,2,1}{c1,c2,c3}{1,3,4,2}{.3}{.25}};
			\draw(3*\x,2.75*\y) node(nl6){\scriptsize $(0,1,1,0)$};
		\draw(1*\x,4*\y) node(n7){\AlphaPerm{1,2,1}{c1,c2,c3}{4,1,2,3}{.3}{.25}};
			\draw(1*\x,3.75*\y) node(nl7){\scriptsize $(3,0,0,0)$};
		\draw(2*\x,4*\y) node(n8){\AlphaPerm{1,2,1}{c1,c2,c3}{3,1,4,2}{.3}{.25}};
			\draw(2*\x,3.75*\y) node(nl8){\scriptsize $(1,0,1,0)$};
		\draw(3*\x,4*\y) node(n9){\AlphaPerm{1,2,1}{c1,c2,c3}{2,3,4,1}{.3}{.25}};
			\draw(3*\x,3.75*\y) node(nl9){\scriptsize $(0,1,1,0)$};
		\draw(1.5*\x,5*\y) node(n10){\AlphaPerm{1,2,1}{c1,c2,c3}{4,1,3,2}{.3}{.25}};
			\draw(1.5*\x,4.75*\y) node(nl10){\scriptsize $(3,0,1,0)$};
		\draw(2.5*\x,5*\y) node(n11){\AlphaPerm{1,2,1}{c1,c2,c3}{3,2,4,1}{.3}{.25}};
			\draw(2.5*\x,4.75*\y) node(nl11){\scriptsize $(1,1,1,0)$};
		\draw(2*\x,6*\y) node(n12){\AlphaPerm{1,2,1}{c1,c2,c3}{4,2,3,1}{.3}{.25}};
			\draw(2*\x,5.75*\y) node(nl12){\scriptsize $(3,1,1,0)$};
		\draw(n1) -- (nl2);
		\draw(n1) -- (nl3);
		\draw(n2) -- (nl4);
		\draw(n2) -- (nl5);
		\draw(n3) -- (nl5);
		\draw(n3) -- (nl6);
		\draw(n4) -- (nl7);
		\draw(n5) -- (nl8);
		\draw(n6) -- (nl9);
		\draw(n7) -- (nl10);
		\draw(n8) -- (nl10);
		\draw(n8) -- (nl11);
		\draw(n9) -- (nl11);
		\draw(n10) -- (nl12);
		\draw(n11) -- (nl12);
		\begin{pgfonlayer}{background}
			\fill[gray!50!white,rounded corners](1.6*\x,.6*\y) -- (2.4*\x,.6*\y) -- (2.4*\x,1.3*\y) -- (1.6*\x,1.3*\y) -- cycle;
			\fill[gray!50!white,rounded corners](1.1*\x,1.6*\y) -- (1.9*\x,1.6*\y) -- (1.9*\x,2.3*\y) -- (1.1*\x,2.3*\y) -- cycle;
			\fill[gray!50!white,rounded corners](2.1*\x,1.6*\y) -- (2.9*\x,1.6*\y) -- (2.9*\x,2.3*\y) -- (2.1*\x,2.3*\y) -- cycle;
			\fill[gray!50!white,rounded corners](.6*\x,2.6*\y) -- (1.4*\x,2.6*\y) -- (1.4*\x,3.3*\y) -- (.6*\x,3.3*\y) -- cycle;
			\fill[gray!50!white,rounded corners](.6*\x,3.6*\y) -- (1.4*\x,3.6*\y) -- (1.4*\x,4.3*\y) -- (.6*\x,4.3*\y) -- cycle;
			\fill[gray!50!white,rounded corners](1.6*\x,2.6*\y) -- (2.4*\x,2.6*\y) -- (2.4*\x,4.3*\y) -- (1.6*\x,4.3*\y) -- cycle;
			\fill[gray!50!white,rounded corners](2.6*\x,2.6*\y) -- (3.4*\x,2.6*\y) -- (3.4*\x,4.3*\y) -- (2.6*\x,4.3*\y) -- cycle;
			\fill[gray!50!white,rounded corners](1.1*\x,4.6*\y) -- (1.9*\x,4.6*\y) -- (1.9*\x,5.3*\y) -- (1.1*\x,5.3*\y) -- cycle;
			\fill[gray!50!white,rounded corners](2.1*\x,4.6*\y) -- (2.9*\x,4.6*\y) -- (2.9*\x,5.3*\y) -- (2.1*\x,5.3*\y) -- cycle;
			\fill[gray!50!white,rounded corners](1.6*\x,5.6*\y) -- (2.4*\x,5.6*\y) -- (2.4*\x,6.3*\y) -- (1.6*\x,6.3*\y) -- cycle;
		\end{pgfonlayer}
	\end{tikzpicture}
	\caption{The weak order on $\Symmetric_{(1,2,1)}$, where the permutations are labeled by their $(1,2,1)$-codes.}
	\label{fig:121_tamari_lattice}
\end{figure}

\section{Lattice paths and various Tamari lattices}
\label{sec:lattice_paths}

Among other things, \cite{muehle19tamari} introduces a bijection $\Theta$ from $\Symmetric_{\alpha}(231)$ to a certain family of northeast paths, denoted by $\Paths_{\nu_{\alpha}}$.  This map was simplified in \cite{ceballos20the} and used to show that the lattice $\Tamari_{\alpha}$ is isomorphic to a certain lattice on $\Paths_{\nu_{\alpha}}$, denoted by $\Tamari_{\nu_{\alpha}}$~\cite{ceballos20the}*{Theorem~II}.  The proof of this result is only partially bijective, and relies on structural properties of both lattices.  

Our parabolic BW-codes will eventually enable us to show that $\Theta$ is in fact an anti-isomorphism from $\Tamari_{\alpha}$ to $\Tamari_{\flip(\nu_{\alpha})}$, where $\flip(\nu_{\alpha})$ is essentially the northeast path $\nu_{\alpha}$ read backwards.  

In order to prove this result, we will take a detour through certain families of lattice paths on $\mathbb{N}\times\mathbb{N}$.  Even though the objects presented here and the bijections relating them mostly belong to the combinatorial folklore and are known to many combinatorialists, we shall present them in form of a unified framework in which their definitions and main properties fit together fairly.

%

\subsection{Dyck paths and the ordinary Tamari lattice}

We first consider \defn{up steps} (of the form $U\defs(1,1)$) and \defn{down
steps} (of the form $D\defs(1,-1)$).  A \defn{Dyck path} of semilength $n$ is
a lattice path using only up and down steps, starting and ending on the
$x$-axis and never going below it.  Consequently, any Dyck path uses as many
up steps as it uses down steps.  Let $\Dyck_{n}$ denote the set of Dyck paths
of semilength $n$.  The \defn{ordinate} of a lattice point on a Dyck path is
simply the value of its $y$-coordinate.

If $P\in\Dyck_{n}$, then any up step $U$ on $P$ has a \defn{matching} down
step: this is the first down step $D$ on $P$ whose starting point has the same
ordinate as the ending point of $U$.  In particular, the portion of $P$
strictly between $U$ and $D$ is a Dyck path (of strictly smaller semilength)
in its own right.

A \defn{valley} of $P$ is a lattice point $V$ on $P$ preceded by a down step
and followed by an up step.  The \defn{rotation} of $P$ by a valley $V$ is the
Dyck path $P'\in\Dyck_{n}$ obtained by swapping the down step before $V$ and
the portion of $P$ (weakly) between the up step after $V$ and its matching
down step.  The reflexive and transitive closure of this operation yields a
partial order on $\Dyck_{n}$.

The set of Dyck paths ordered by this \defn{rotation order} forms the
(ordinary) \defn{Tamari lattice} $\Tamari_{n}$, first described in
\cite{tamari62algebra}.

\begin{remark}\label{rem:ordinary_tamari_as_alpha_tamari}
It may not be immediately clear from this definition, but the ordinary Tamari
lattice is a particular instance of an $\alpha$-Tamari lattice (see
Section~\ref{sec:weak_order}), namely when $\alpha=(1,1,\ldots,1)$; see
\cite{bjorner97shellable}*{Section~9}.
\end{remark}

\subsection{Northeast paths and the $\nu$-Tamari lattice}

Now, we consider \defn{north steps} (of the form $N\defs(0,1)$) and \defn{east
steps} (of the form $E\defs(1,0)$).  A \defn{northeast path} of length $n$ is
a lattice path using $k$ north and $n-k$ east steps which starts on the
$x$-axis.  The \defn{height} of a lattice point on a northeast path is the
value of its $y$-coordinate.  A \defn{valley} of a northeast path $\mu$ is a
lattice point $V$ on $\mu$ preceded by an east step and followed by a north
step.

Given a northeast path $\nu$, a \defn{$\nu$-path} is a northeast path which
shares the starting and ending points with $\nu$ and never goes below $\nu$.
Let $\Paths_{\nu}$ denote the set of $\nu$-paths.  The \defn{horizontal
distance} of a lattice point $Q$ on $\mu\in\Paths_{\nu}$, denoted by
$\horiz_{\nu}(Q)$, is the largest number of east steps that can be added to
$Q$ without crossing to the other side of $\nu$.

\begin{figure}
  \centering
  \includegraphics[page=3, width=0.8\textwidth]{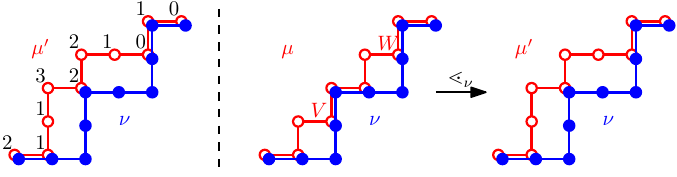}
  \caption{Example of horizontal distance and cover relation in $\Tamari_\nu$ for northeast paths.}
  \label{fig:nu-tamari-cover}
\end{figure}

The \defn{rotation} of $\mu$ by a valley $V$ is the path $\mu'\in\Paths_{\nu}$
obtained by exchanging the east step before $V$ with the portion of $\mu$
between $V$ and the next lattice point $W$ on $\mu$ satisfying
$\horiz_{\nu}(V)=\horiz_{\nu}(W)$.  In this situation, we write
$\mu\lessdot_{\nu}\mu'$.  See Figure~\ref{fig:nu-tamari-cover} for an
illustration.  The reflexive and transitive closure of this operation yields a
partial order on $\Paths_{\nu}$.

For any northeast path $\nu$, the set of $\nu$-paths ordered by this
\defn{rotation order} forms the \defn{$\nu$-Tamari lattice} $\Tamari_{\nu}$
introduced in \cite{preville17enumeration}.

\begin{remark}
Note that the definition that $\Tamari_{\nu}\cong\Tamari_{n}$ when
$\nu=(NE)^{n}$. The isomorphism is given by substituting north steps by up
steps and east steps by down steps.
\end{remark}

In \cite{ceballos20nu}, it was shown that the $\nu$-Tamari lattice can be
realized using the componentwise order on so-called $\nu$-bracket vectors.  If
$\nu$ is a northeast path of length $n$, then the \defn{minimal $\nu$-bracket
vector} is the vector $\bvec^{\min}$ consisting of $n+1$ entries, whose $i$-th
entry is the height of the $i$-th lattice point on $\nu$.  If $\nu$ has $k$
north steps, then the \defn{fixed positions} are the entries
$f_{0},f_{1},\ldots,f_{k}$, where $f_{i}$ is the position of the last
appearence of $i$ in $\bvec^{\min}$.  An integer vector $\bvec$ with $n+1$
entries is a \defn{$\nu$-bracket vector}, if it has the following three
properties:

\begin{description}
\item[B1\label{it:bvec1}] for $0 \leq s \leq k$, we have $\bvec(f_s) = s$;
\item[B2\label{it:bvec2}] for $1 \leq i \leq n+1$, we have $\bvec^{\min}(i) \leq \bvec(i) \leq k$;
\item[B3\label{it:bvec3}] if $\bvec(i) = s$, then for all $j$ with $i < j < f_s$, we have $\bvec(j) \leq s$.
\end{description}
The set of $\nu$-bracket vectors is denoted by $\bvecset_{\nu}$. The following theorem was proven by means of an explicit bijection in \cite{ceballos20nu}*{Section~4}.


\begin{theorem}[\cite{ceballos20nu}*{Theorem~21}]
\label{thm:nu_tamari_bracket_vectors}
For any northeast path $\nu$, the $\nu$-Tamari lattice $\Tamari_{\nu}$ is
isomorphic to $\bigl(\bvecset_{\nu},\comporder\bigr)$.
\end{theorem}

\subsection{$\nu$-Tamari lattices are intervals of ordinary Tamari lattices}
	\label{sec:nu_tamari_intervals}
It was shown in \cite{preville17enumeration}*{Theorem~3} in terms of binary
trees that every $\nu$-Tamari lattice is isomorphic to an interval in some
ordinary Tamari lattice. More precisely, the bijection in \cite[Section~2~and~3]{preville17enumeration} from binary trees to pairs of non-crossing lattice paths is done by a double reading of a word obtained from the contour of the binary tree, and the reverse direction is more complicated, involving a so-called ``push-gliding'' algorithm. Later, this bijection is transplanted from binary trees to Dyck paths in \cite[Section~2]{fang2017the}, but only the direction from Dyck paths to pairs of non-crossing lattice paths, not the reverse direction, and without proof. Furthermore, the bijection given in \cite{fang2017the} needs to keep track of how pairs of peaks appear in sub-paths of the Dyck path, making it seem non-trivial and hard to reverse.

We now propose a reformulation of the same bijection, given in both ways, along with proofs of equivalence to the original bijection. Our reformulation is much simpler than the original ones, as the bijectivity is trivial, and does not require any complicated algorithm in both directions.

Clearly, a northeast path is uniquely determined by the lengths of the
\defn{runs} of east steps at each height. In other words, if $\nu$ is a
northeast path, then we can write it uniquely as
\begin{displaymath}
	\nu = E^{a_{0}} N E^{a_{1}} N \cdots E^{a_{k-1}} N E^{a_{k}}.
\end{displaymath}
By abuse of notation we will also write $\nu=[a_{0},a_{1},\ldots,a_{k}]$.
Now, if $\mu\in\Paths_{\nu}$ with $\mu=[b_{0},b_{1},\ldots,b_{k}]$, then we
have $\sum_{i=0}^k a_i = \sum_{i=0}^k b_i = m$ and $\sum_{i=0}^j a_i \geq
\sum_{i=0}^j b_i$ for all $0 \leq j \leq k$. We may have $a_i=0$ or $b_i=0$
for some indices $i$. 

\begin{construction}
\label{constr:dyck-bij}
Given a northeast path $\nu$ composed of $k$ north steps and $n-k$ east steps,
let $\mu \in \Paths_\nu$ such that $\nu = [a_0, a_1, \ldots, a_k]$ and $\mu =
[b_0, b_1, \ldots, b_k]$.  We define the Dyck path $\ddyck(\nu, \mu)$ of
semilength $n+1$ by
\begin{displaymath}
\ddyck(\nu, \mu) \defs  U^{a_0 + 1} D^{b_0 + 1} \cdots U^{a_k + 1} D^{b_k + 1}.
\end{displaymath}
	
Conversely, if $P\in\Dyck_{n+1}$ we can recover the pair $(\nu,\mu)$
satisfying $P=\ddyck(\nu,\mu)$ by looking at the lengths of the runs of up and
down steps in $P$, which determine the $a_{i}$'s and $b_{i}$'s, respectively.
\end{construction}

The map $\ddyck$ is a bijection from $\Dyck_{n+1}$ to the set 
\begin{displaymath}
	\bigl\{(\nu,\mu)\mid \nu\;\text{has length}\;n\;\text{and}\;\mu\in\Paths_{\nu}\bigr\},
\end{displaymath}
and is illustrated in Figure~\ref{fig:nutodyck}. The map $\ddyck$ is simple and clearly bijective, a quality absent from previous formulations.

\begin{figure}
	\centering
	\includegraphics[page=1,width=\textwidth]{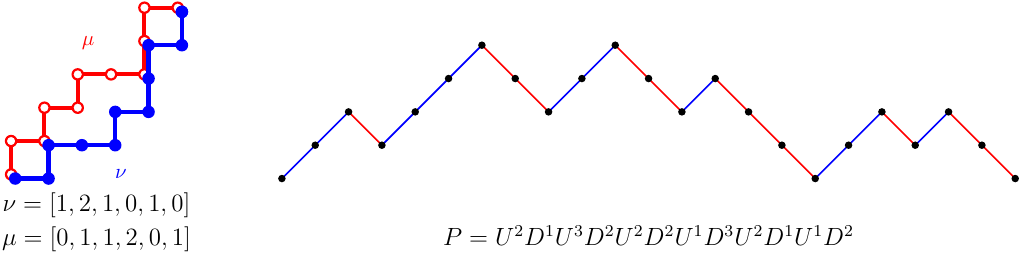}
	\caption{An example of the bijection $\ddyck$.} \label{fig:nutodyck}
\end{figure}

\begin{proposition}
\label{prop:nu-tamari-cover}
Let $\mu, \mu'\in\Paths_\nu$, with $\nu = [a_0, a_{1},\ldots, a_k]$,
$\mu = [b_0, b_{1},\ldots, b_k]$ and $\mu' = [b_0', b'_{1},\ldots, b_k']$.
Then $\mu\lessdot_{\nu}\mu'$ if, and only if $b_i = b_i'$ except for two
indices $\ell < m$ such that $b_\ell > 0$ and that $m$ is the first index
after $\ell$ satisfying $\sum_{i = \ell+1}^m (b_i - a_i) \geq 0$. In this
case, we have $b_\ell' = b_\ell - 1$ and $b_m' = b_m + 1$.
\end{proposition}

\begin{proof}
We observe that the minimum of $\horiz_\nu$ for points with the same height
occurs at the rightmost point. For any index $0 \leq j \leq k$, let $V_j$ be
the rightmost point of height $j$ on $\mu$. We have $\horiz_\nu(V_j) = \sum_{i
= 0}^j a_i - \sum_{i=0}^j b_i$.

Assume that $\mu'$ covers $\mu$ in $\Tamari_\nu$, with $V$ the valley of
$\mu$ leading to this covering relation, and $W$ the first lattice point after
$V$ with $\horiz_\nu(W) = \horiz_\nu(V)$. Let $\ell$ (resp. $m$) be the height
of $V$ (resp. $W$). We have $V = V_\ell$, also $b_\ell > 0$ as $V$ is a
valley. By the definition of $W$, we know that $\ell < m$, $\horiz_\nu(V) =
\horiz_\nu(W) \geq \horiz_\nu(V_m)$, and $\horiz_\nu(V) < \horiz_\nu(V_{m'})$
for any $m'$ such that $\ell < m' < m$. Therefore, $m$ is the first index such
that $\ell < m$ and $\horiz_\nu(V_\ell) \geq \horiz_\nu(V_m)$. In other
words,
\begin{displaymath}
0\leq \horiz_{\nu}(V_{\ell})-\horiz_{\nu}(V_{m}) =
\sum_{i=0}^{\ell}(a_{i}-b_{i}) - \sum_{i=0}^{m}(a_{i}-b_{i}) =
\sum_{i=\ell+1}^{m}(a_{i}-b_{i}).
\end{displaymath}
In this case, we also have
\begin{displaymath}
	\mu' = [b_0, \ldots, b_{\ell-1}, b_\ell + 1, b_{\ell+1}, \ldots, b_{m-1}, b_m - 1, b_{m+1}, \ldots, b_k]
\end{displaymath}
as desired.

Conversely, suppose that we have two indices $\ell$ and $m$ satisfying the
given conditions. Let $V=V_\ell$ on $\mu$ and $W$ be the first lattice point
after $V$ with $\horiz_\nu(W) = \horiz_\nu(V)$. By the conditions that $m$
satisfies, for any $m'$ such that $\ell < m' < m$, we must have
$\horiz_\nu(V_\ell) < \horiz_\nu(V_{m'})$, meaning that the height of $W$ is
at least $m$. We also observe that the horizontal distance of the leftmost
point of height $m$ is $\horiz_\nu(V_{m-1}) + a_m > \horiz_\nu(V_\ell)$, while
$\horiz_\nu(V_m) \leq \horiz_\nu(V_\ell)$ for $V_m$ the rightmost point of
height $m$. Therefore, the height of $W$ must be $m$. We conclude the proof by
plugging in the definition of $\horiz_\nu(V_j)$ in the condition.
\end{proof}

The following proposition show that $\ddyck$ is the same bijection as that in \cite{preville17enumeration}, but from non-crossing lattice paths instead of binary trees. 
\begin{proposition}
\label{prop:dyck-bij}
Given a northeast path $\nu$ of length $n$, the map $\ddyck$ is an isomorphism
from $\Tamari_\nu$ to an interval $I_\nu$ of $\Tamari_{n+1}$.
\end{proposition}

\begin{proof}
It is clear that $\ddyck(\nu,\mu)$ is a Dyck path of semilength $n+1$ for
every $\mu\in\Paths_{\nu}$. We already know from
\cite{preville17enumeration}*{Theorem~3} that $\Tamari_\nu$ is isomorphic to
some interval $I_\nu$ in $\Tamari_{n+1}$. We only need to show that this
interval is exactly $\bigl\{\ddyck(\nu,\mu)\mid \mu\in\Paths_{\nu}\bigr\}$
under ordinary rotation order.  To that end, we show that if $\mu'$ covers
$\mu$ in $\Tamari_\nu$, then $P' = \ddyck(\nu,\mu')$ covers $P =
\ddyck(\nu,\mu)$ in $\Tamari_{n+1}$.
  
By abuse of notation, if we have $P = D^{c_0} U D^{c_1} U \cdots U
D^{c_{n+1}}$ for some $(c_i)_{0 \leq i \leq n+1}$, then we also write $P =
[c_0, c_1, \ldots, c_{n+1}]$. We note that $c_0$ is always $0$.

Now pick an arbitrary northeast path $\nu$ of length $n$ and some
$\mu\in\Paths_{\nu}$. Suppose that $\nu = [a_0, a_{1},\ldots, a_k]$ and $\mu =
[b_0, b_{1},\ldots, b_k]$. For $d\in\{0,1,\ldots,k\}$ we define
\begin{displaymath}
f_{\nu}(d) \defs d + \sum_{i=0}^d a_i. 
\end{displaymath}
Then, $f_{\nu}(d)$ describes the number of steps on $\nu$ before the
$(d+1)$-st north step.
  
Let $P=\ddyck(\nu,\mu)$.  If $P = [c_0, c_1, \ldots, c_n]$, then each run of
east steps in $\nu$ contributes some entry $c_{i}=0$.  Since we add one more
up step to $P$ than such a run contains east steps, it follows that $c_i = 0$
except when $i = f_{\nu}(j)+1$ for some $0 \leq j \leq k$, and in such case,
we have $c_{f_{\nu}(j)} = b_j + 1$. The same holds for $\mu' = [b_0',
b'_{1},\ldots, b_k']$ and $P' = [c_0', c_1', \ldots, c_n']$. 
	
Since $\mu'$ covers $\mu$ in $\Tamari_\nu$, denote by $\ell$ and $m$ the two
indices satisfying the conditions of Proposition~\ref{prop:nu-tamari-cover}.
Now, let $\ell^* = f_\nu(\ell)+1$ and $m^* = f_\nu(m)+1$. We now show that
$\ell^*$ and $m^*$ are indices satisfying the conditions of
Proposition~\ref{prop:nu-tamari-cover} for $P$ and $P'$ in
$\Tamari_{(NE)^{n+1}}$. It is clear that $\ell^* < m^*$, and from the
definition of $\ddyck$, we know that $c_i = c_i'$ except for $\ell^*$ and
$m^*$. We also have $c_{\ell^*}' = b_\ell' + 1 = b_\ell = c_{\ell^*} - 1$.
Similarly, $c_{m^*}' = c_{m^*} + 1$.
    
For $j^{*}>\ell^{*}$, let $s(j^{*})\defs\sum_{i=\ell^{*}+1}^{j^{*}}(c_{i}-1)$.
To guarantee that $P'$ covers $P$ in $\Tamari_{n+1}$, we transport these paths
to $\Tamari_{(NE)^{n+1}}$ and use Proposition~\ref{prop:nu-tamari-cover}. We
thus only need to show that $m^*$ is the first index after $\ell^*$ such that
$s(m^{*})\geq 0$, as $(NE)^{n+1} = [0, 1, \ldots, 1]$. Suppose that $j^*$ is
the smallest index with $j^* > \ell^*$ and $s(j^{*}) \geq 0$. If there is no
index $j$ such that $j^* = f_{\nu}(j)+1$, then $c_{j^*} = 0$ and thus
$s(j^{*}-1)\geq s(j^{*})\geq 0$, contradicting the minimality of $j^*$. We
thus have $j^* = f_{\nu}(j)+1$ for some $j$, and we observe that
\begin{align*}
	s(j^{*}) & = \sum_{i = \ell^* + 1}^{j^*} (c_i - 1) = \left(\sum_{i = \ell + 1}^{j} (b_i + 1)\right) - (j^* - \ell^*)\\
	& = \left(\sum_{i=\ell+1}^{j} b_{i}\right) - \underbrace{\Bigl((j^{*}-\ell^{*})-(j - \ell)\Bigr)}_{=\lvert\{i\mid \ell^{*}<i\leq j^{*}, c_{i}=0\}\rvert}  = \sum_{i=\ell + 1}^{j} (b_i - a_i).
\end{align*}
The last equality follows from the fact that $c_{i}=0$ indicates two
consecutive up steps in $P$ and therefore correspond to an east step in $\nu$.
Similarly, we have $s(m^*) = \sum_{i=\ell+1}^{m} (b_i - a_i)$, and by
Proposition~\ref{prop:nu-tamari-cover}, we have $s(m^{*}) \geq 0$. By the
minimality of $j^*$, we have $j^* \leq m^*$.
  
As $j$ satisfies the sum condition of Proposition~\ref{prop:nu-tamari-cover}
with respect to $\ell$, we must have $m\leq j$. This implies $m^{*}\leq
j^{*}$, and we obtain $j^{*}=m^{*}$.  We conclude that $P'$ covers $P$ in
$\Tamari_{n+1}$.
\end{proof}

\subsection{Two sequence statistics of Dyck paths}
	\label{sec:dyck_sequence_stats}

In preparation of things to come, we now consider an anti-isomorphism of
$\Tamari_{n}$ that exchanges two particular sequence statistics on Dyck paths. 
While this property is known to experts, we could not find an explicit reference stating it in the framework of Dyck paths.  It can be deduced, however, from the recent work of Pons on Tamari interval posets~\cite{pons19intervalposets}; see the paragraph after Theorem~23 therein.  We translate the appropriate specialization of her result to the framework of Dyck paths.
We then relate this map and these statistics to $\nu$-Tamari lattices via the map $\ddyck$. 

A \defn{rising contact} of a Dyck path $P$ is an up step in $P$ that starts
on the $x$-axis. Every Dyck path $P$ of length $n > 0$ can be uniquely
decomposed into $P = P_\ell U P_r D$ with $P_\ell, P_r$ both Dyck paths, by
taking $P_\ell$ to be the sub-path before the last rising contact of $P$. We
denote by $\epsilon$ the empty Dyck path, and we define an involution $\conj$
recursively by
\begin{align}
\label{eq:conj-def}
\conj(\epsilon) \defs \epsilon, && \conj(P_\ell U P_r D) \defs \conj(P_r) U \conj(P_\ell) D.
\end{align}

The following is well-known and can be proven by induction using the
decomposition $P = P_\ell U P_r D$.

\begin{proposition}
\label{prop:tamari-inv}
The involution $\conj$ is an anti-isomorphism of $\Tamari_n$. 
\end{proposition}

We define $\conj' \defs \ddyck^{-1} \circ \conj \circ \ddyck$, which is simply
$\conj$ conjugated to the domain of $\nu$-paths using $\ddyck$. See
Figure~\ref{fig:conj} for an example of $\conj$ and $\conj'$.

Given a northeast path $\nu$, we denote by $\flip(\nu)$ the northeast path
obtained by reversing $\nu$ and exchanging north and east steps.
Geometrically, $\flip(\nu)$ is $\nu$ reflected across a diagonal of slope
$-1$. It is known that, for $(\nu', \mu') = \conj'(\nu, \mu)$, we have $\nu =
\flip(\nu)$ (see \cite[Theorem~2~and~3]{preville17enumeration}).
We have the following corollary.

\begin{figure}
  \centering
  \includegraphics[page=5,width=\textwidth]{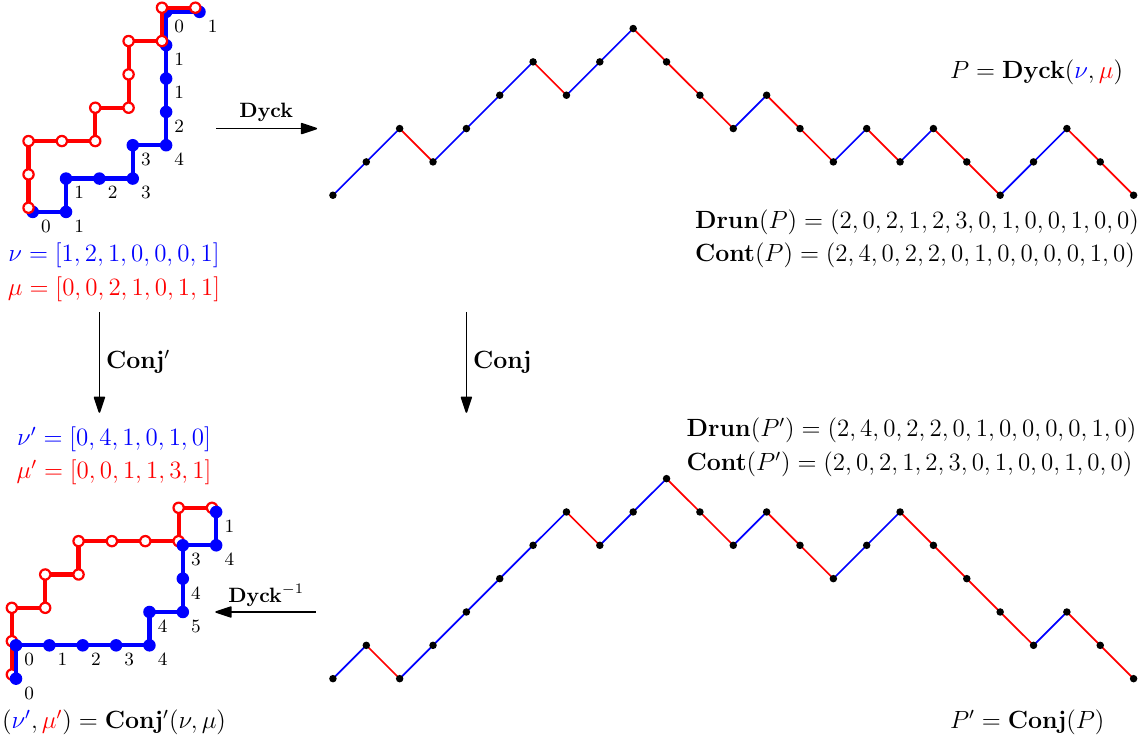}
  \caption{Example of the bijections $\conj$ and $\conj'$ and how they
transform the statistics $\cont$ and $\drun$. The reverse horizontal distances
of points on the lower paths are also given.}
  \label{fig:conj}
\end{figure}

\begin{corollary} \label{coro:tamari-inv}
The bijection $\conj'$ is an anti-isomorphism between $\Tamari_{v}$ and
$\Tamari_{\flip(v)}$.
\end{corollary}


We now consider two statistics on Dyck paths that are interchanged by $\conj$.
Given a Dyck path $P$, we define its \defn{descent run sequence}, denoted by
$\drun(P)$ and \emph{indexed from $0$ to $n$}, as follows. We write $P =
D^{c_0} U D^{c_1} U \cdots U D^{c_n}$ with some $c_i \geq 0$ (again, noting
$c_0 = 0$), then we take $\drun(P)_i = c_{n-i}$ for all $0 \leq i \leq n$. The
map $\drun$ is injective because $P$ can be reconstructed from the $c_i$'s.

We now define another sequence statistic.
The \defn{contact sequence} of $P$, denoted by $\cont(P)$ and \emph{indexed
from $0$ to $n$}, is obtained by taking $\cont(P)_0$ the number of rising
contacts of $P$, and $\cont(P)_i$ the number of rising contacts of the
sub-Dyck path strictly between the $i$-th up step and its matching down step.

Examples for $\cont$ and $\drun$ are given in Figure~\ref{fig:conj}. The
following result, also illustrated in Figure~\ref{fig:conj}, is well-known. In
terms of binary trees, $\cont$ (resp. $\drun$) describes maximal left (resp.
right) descending paths, and the counterpart of $\conj$ on binary trees is
taking the vertical mirror image. 
The following can be proven inductively.

\begin{proposition}\label{prop:drun-cont}
  For any Dyck path $P$, we have $\cont(P) = \drun\bigl(\conj(P)\bigr)$.
\end{proposition}

As $\ddyck$ is bijective, let $(\nu, \mu) = \ddyck^{-1}(P)$, and define
$\drun(\nu, \mu) = \drun(P)$ and $\cont(\nu, \mu) = \cont(P)$. Suppose that
$\nu = [a_0, a_1, \ldots, a_k]$ and $\mu = [b_0, b_1, \ldots, b_k]$. Then we
have
\begin{equation}
\label{eq:drun-path-pair}
\drun(\nu, \mu) = \drun(P) = (b_k+1, 0^{a_k}, b_{k-1}+1, 0^{a_{k-1}}, \ldots, b_0+1, 0^{a_0}, 0).
\end{equation}
Here ``$0^{a}$'' means the entry $0$ repeated $a$ times.  For an expression of
$\cont$, we need some more definitions.
We define the \defn{reverse horizontal distance} to $\mu$ of a lattice point
$Q$ on $\nu$, denoted by $\horiz'_\mu(Q)$, to be the number of west steps
$(-1,0)$ we can take from $Q$ before crossing to the other side of $\mu$. See
Figure~\ref{fig:conj} for an example.  This figure also illustrates the
following result.

\begin{proposition} \label{prop:cont-path-pair}
Let $P = \ddyck(\nu, \mu)$ with $\nu, \mu$ northeast paths of length $n$. Take
the sequence $(d_i)_{0 \leq i \leq n}$ with $d_i = \horiz'_\mu(Q_i)$, where
$Q_i$ is the $(i+1)$-st lattice point of $\mu$. Then we have
\begin{align}
\begin{split} \label{eq:cont-path-pair}
\cont(\nu, \mu)_0 &= \bigl\lvert \{ \ell \mid 0 \leq \ell \leq n, d_\ell = 0 \}\bigr\rvert; \\
\cont(\nu, \mu)_i &= \bigl\lvert \{ \ell \mid i < \ell \leq n, d_\ell = d_i + 1, \forall i < m \leq \ell, d_m > d_i \}\bigr\rvert.
\end{split}
\end{align}
\end{proposition}

\begin{proof}
Assume that $\nu = [a_0, a_1, \ldots, a_k]$ and $\mu = [b_0, b_1, \ldots,
b_k]$. For each $0 \leq i \leq k$, there are $a_i + 1$ lattice points of
height $i$ on $\nu$, and the $(i+1)$-st consecutive run of up steps of $P$
consists of $a_i+1$ steps.
According to the construction of $\ddyck$, the $(i+1)$-st up step in $P$
corresponds to a lattice point $Q_{i}$ on $\nu$.  Moreover, $Q_{i}$ is the
leftmost point with height $j$ on $\nu$ if, and only if the corresponding up
step is the first in the $j$-th run of up steps in $P$.
  
We now prove that $d_i$ is the ordinate of the starting point of the
$(i+1)$-st up step in $P$. We proceed by induction. For $i = 0$, we have $d_0
= 0 = \horiz'_\mu(Q_0)$, as $\nu$ and $\mu$ start at the origin. Now suppose
that $d_{i-1}$ is the ordinate of the starting point of the $i$-th up step in
$P$. We have two cases.
\begin{enumerate}[(i)]
\item If $Q_i$ is not the leftmost lattice point with its height, then its
corresponding up step in $P$ directly comes after the one associated with
$Q_{i-1}$. Thus, the ordinate of the starting point of this up step is
$d_{i-1} + 1 = d_i$.
\item If $Q_i$ is the leftmost lattice point with height $\ell$, then its
corresponding up step is the first up step of the next run after the up step
associated with $Q_{i-1}$. The ordinate of the up step associated with $Q_i$
is thus $d_{i-1} + 1 - (b_\ell + 1)$, taking into account the length of the
$\ell$-th run of down steps in $P$. This is also equal to $d_i$, as we reach
$Q_i$ from $Q_{i-1}$ by a north step, and it takes $b_\ell$ less west steps to
cross $\mu$ from $Q_i$ than from $Q_{i-1}$.
\end{enumerate}
We thus conclude the induction. Now \eqref{eq:cont-path-pair} is a
translation of the definition of $\cont$ in terms of the $d_i$'s.
\end{proof}

\section{The $\alpha$-Tamari lattices are certain $\nu$-Tamari lattices}
\label{sec:alpha_tamari_lattice_isomorphism}

It was shown in \cite{ceballos20the}*{Theorem~II}, that the $\alpha$-Tamari lattice is isomorphic to a certain $\nu$-Tamari lattice, namely, when $\nu$ is the \defn{$\alpha$-bounce path}, defined by
\begin{equation}\label{eq:alpha_bounce_path}
	\nu_{\alpha} \defs N^{\alpha_{1}}E^{\alpha_{1}}N^{\alpha_{2}}E^{\alpha_{2}}\cdots N^{\alpha_{r}}E^{\alpha_{r}},
\end{equation}
where $\alpha=(\alpha_{1},\alpha_{2},\ldots,\alpha_{r})$.  The proof given in
\cite{ceballos20the} is rather indirect and exploits certain lattice-theoretic
properties of the $\alpha$- and the $\nu$-Tamari lattices. In this section,
we give a direct proof using the realization of $\Tamari_{\alpha}$ as the
componentwise order on $\alpha$-codes established in
Theorem~\ref{thm:alpha_tamari_lattice_codes}.

In particular, we construct a bijection from the $\alpha$-codes to the
$\nu_{\alpha}$-bracket vectors.  Let us adapt the definitions to the
particular case of $\nu_{\alpha}$.  The \defn{minimal $\nu_\alpha$-bracket
vector}, denoted by $\bvec^{\min}_{\alpha}$, is defined by
\begin{displaymath}
	\bvec^{\min}_{\alpha}(k) \defs \begin{cases}
		i+s_{a-1}-1, & \text{if}\;k=2s_{a-1}+i\;\text{for}\;0<i\leq\alpha_{a},\\
		s_{a}, & \text{if}\;k=2s_{a-1}+\alpha_{a}+i\;\text{for}\;0<i\leq\alpha_{a},\\
		n, & \text{if}\;k=2n+1.
	\end{cases}
\end{displaymath}
We write $\bvecset_{\alpha}$ instead of $\bvecset_{\nu_{\alpha}}$ for the set of $\nu_{\alpha}$-bracket vectors.

For $\bvec \in \bvecset_\alpha$, by \eqref{it:bvec1}, there are $n+1$
positions with fixed value in a vector of length $2n+1$. For simplification,
we define a reduced version of $\nu_\alpha$-bracket vectors. For
$\bvec\in\bvecset_{\alpha}$, we define its \defn{$\nu_\alpha$-reduced vector}
$\rvec$ by 
\begin{displaymath}
	\rvec(s_{a-1} + i) \defs \bvec(2s_{a-1} + \alpha_{a} + i)
\end{displaymath}
for $1 \leq i \leq \alpha_a$.  It is clear that $\rvec$ is obtained from
$\bvec$ by removing components whose indices are fixed positions. To recover $\bvec$ from
$\rvec$, we only need to fill in the positions of the fixed positions according to
\eqref{it:bvec1}. Let $\redbvec$ denote the ``reduction'' map from $\bvec$ to $\rvec$, and let $\extrvec$ be its inverse.

Such $\nu_\alpha$-reduced vectors thus inherit the following
properties from $\nu_\alpha$-bracket vectors.

\begin{proposition} \label{prop:reduced-def}
A vector $\rvec\in\mathbb{N}^{n}$ is a $\nu_\alpha$-reduced vector if, and
only if:
\begin{description}
\item[R1\label{it:rvec1}] for $1 \leq i \leq n$, we have $s_{\indicator(i)} \leq \rvec(i) \leq n$;
\item[R2\label{it:rvec2}] for all $i, j$ with $i < j \leq s_{\indicator(\rvec(i)+1)-1}$, we have $\rvec(j) \leq \rvec(i)$.
\end{description}
\end{proposition}

\begin{proof}
Let $\bvec$ be the $\nu_\alpha$-bracket vector corresponding to $\rvec$. We
only need to show that the conditions for $\rvec$ are equivalent to those for
$\bvec$. 
	
Condition~\eqref{it:bvec1} for $\bvec$ is satisfied by construction. The
equivalence between \eqref{it:bvec2} for $\bvec$ and \eqref{it:rvec1} for
$\rvec$ is trivial given the definition of $\bvec^{\min}_\alpha$. 
	
Now for the equivalence between \eqref{it:bvec3} for $\bvec$ and
\eqref{it:rvec2} for $\rvec$, we observe that for \eqref{it:bvec3} to hold for
$\bvec$, for each $i$ with $\bvec(i) = k$, we only need to check for all $j$
with $i < j \leq 2s_{\indicator(k+1)-1}$, since all indices from
$2s_{\indicator(k+1)-1}+1$ to $f_k$ are fixed positions.
\end{proof}

We can thus take Proposition~\ref{prop:reduced-def} as the definition of
$\nu_\alpha$-reduced vectors without passing through $\nu_\alpha$-bracket
vectors, and we denote by $\rvecset_\alpha$ the set of all
$\nu_\alpha$-reduced vectors. 
By Proposition~\ref{prop:reduced-def},
$(\rvecset_\alpha, \comporder)$ is isomorphic to the $\nu_{\alpha}$-Tamari
lattice. We also have the following property.

\begin{proposition} \label{prop:reduced-decrease}
Given a $\nu_\alpha$-reduced vector $\rvec$, for any indices $i < j$ with
$\indicator(i) = \indicator(j)$, we have $\rvec(i) \geq \rvec(j)$.
\end{proposition}

\begin{proof}
Let $k = \rvec(i)$. By \eqref{it:rvec1}, we have $k \geq s_{\indicator(i)}$,
and thus $s_{\indicator(k+1)} - 1 \geq s_{\indicator(i)+1} - 1 \geq
s_{\indicator(i)}$. Since $\indicator(i) = \indicator(j)$, we have $i < j \leq
s_{\indicator(i)} \leq s_{\indicator(k+1)} - 1$. Then \eqref{it:rvec2} in
Proposition~\ref{prop:reduced-def} states that $\rvec(j) \leq k = \rvec(i)$.
\end{proof}

For any composition $\alpha$, we define a transform $\rvectocode$ on
$\rvecset_\alpha$ such that 
\begin{displaymath}
	\bigl(\rvectocode(\rvec)\bigr)_i \defs \rvec(2s_{\indicator(i)} - \alpha_{\indicator(i)} - i + 1) - s_{\indicator(i)}. 
\end{displaymath}
More intuitively, to
obtain $\rvectocode(\rvec)$, we first split $\rvec$ into regions according to
$\alpha$, then reverse each region while subtracting $s_k$ on the $k\th$
region. We denote by $\codetorvec$ its inverse. Although both $\rvectocode$
and $\codetorvec$ depend on $\alpha$, the composition $\alpha$ should always
be clear from the context. The transformation $\rvectocode$ relates
$\alpha$-codes with $\rvecset_{\alpha}$.

Figure~\ref{fig:code_to_bracket} illustrates the map $\codetorvec$ on the
$(2,3,2,1)$-code $(2,6,0,1,3,1,1,0)$, and Table~\ref{tab:code_to_bracket_121} illustrates this bijection for $\alpha=(1,2,1)$.

\begin{figure}
	\centering
	\begin{tikzpicture}\small
		\def\x{1};
		\draw(1*\x,1*\x) node{$(2,6,0,1,3,1,1,0)$};
			\draw(1*\x,.5*\x) node[rotate=270]{\tiny $\in$};
			\draw(1*\x,.25*\x) node{$\Codes_{(2,3,2,1)}$};
		\draw[thick,c1](0*\x,.8*\x) -- (0*\x,.7*\x) -- (.25*\x,.7*\x) -- (.25*\x,.8*\x);
		\draw[thick,c2](.55*\x,.8*\x) -- (.55*\x,.7*\x) -- (1.15*\x,.7*\x) -- (1.15*\x,.8*\x);
		\draw[thick,c3](1.45*\x,.8*\x) -- (1.45*\x,.7*\x) -- (1.7*\x,.7*\x) -- (1.7*\x,.8*\x);
		\draw[thick,c4](2*\x,.8*\x) -- (2*\x,.7*\x);
		\draw(2.75*\x,1.1*\x) node{$\overset{\codetorvec}{\longrightarrow}$};
		\draw(4.5*\x,1*\x) node{$(8,4,8,6,5,8,8,8)$};
			\draw(4.5*\x,.5*\x) node[rotate=270]{\tiny $\in$};
			\draw(4.5*\x,.25*\x) node{$\rvecset_{(2,3,2,1)}$};
		\draw[thick,c1](3.5*\x,.8*\x) -- (3.5*\x,.7*\x) -- (3.75*\x,.7*\x) -- (3.75*\x,.8*\x);
		\draw[thick,c2](4.05*\x,.8*\x) -- (4.05*\x,.7*\x) -- (4.65*\x,.7*\x) -- (4.65*\x,.8*\x);
		\draw[thick,c3](4.95*\x,.8*\x) -- (4.95*\x,.7*\x) -- (5.2*\x,.7*\x) -- (5.2*\x,.8*\x);
		\draw[thick,c4](5.5*\x,.8*\x) -- (5.5*\x,.7*\x);
		\draw(6.25*\x,1.1*\x) node{$\overset{\extrvec}{\longrightarrow}$};
		\draw(9.5*\x,1*\x) node{$(\mathbf{0},\mathbf{1},8,4,\mathbf{2},\mathbf{3},\mathbf{4},8,6,5,\mathbf{5},\mathbf{6},8,8,\mathbf{7},8,\mathbf{8})$};
			\draw(9.5*\x,.5*\x) node[rotate=270]{\tiny $\in$};
			\draw(9.5*\x,.25*\x) node{$\bvecset_{(2,3,2,1)}$};
		\draw[thick,c1](7.75*\x,.8*\x) -- (7.75*\x,.7*\x) -- (8*\x,.7*\x) -- (8*\x,.8*\x);
		\draw[thick,c2](9.2*\x,.8*\x) -- (9.2*\x,.7*\x) -- (9.8*\x,.7*\x) -- (9.8*\x,.8*\x);
		\draw[thick,c3](10.7*\x,.8*\x) -- (10.7*\x,.7*\x) -- (10.95*\x,.7*\x) -- (10.95*\x,.8*\x);
		\draw[thick,c4](11.55*\x,.8*\x) -- (11.55*\x,.7*\x);
	\end{tikzpicture}
	\caption{Illustration of the map $\codetorvec$.}
	\label{fig:code_to_bracket}
\end{figure}

\begin{table}
	\centering
	\begin{tabular}{c|c|c}
		$\cvec\in\Codes_{(1,2,1)}$ & $\codetorvec(\cvec)\in\rvecset_{(1,2,1)}$ & $\extrvec\circ\codetorvec(\cvec)\in\bvecset_{(1,2,1)}$ \\
		\hline\hline
		$(0,0,0,0)$ & $(1,3,3,4)$ & $(\mathbf{0},1,\mathbf{1},\mathbf{2},3,3,\mathbf{3},4,\mathbf{4})$\\
		$(1,0,0,0)$ & $(2,3,3,4)$ & $(\mathbf{0},2,\mathbf{1},\mathbf{2},3,3,\mathbf{3},4,\mathbf{4})$\\
		$(0,0,1,0)$ & $(1,4,3,4)$ & $(\mathbf{0},1,\mathbf{1},\mathbf{2},4,3,\mathbf{3},4,\mathbf{4})$\\
		$(2,0,0,0)$ & $(3,3,3,4)$ & $(\mathbf{0},3,\mathbf{1},\mathbf{2},3,3,\mathbf{3},4,\mathbf{4})$\\
		$(1,0,1,0)$ & $(2,4,3,4)$ & $(\mathbf{0},2,\mathbf{1},\mathbf{2},4,3,\mathbf{3},4,\mathbf{4})$\\
		$(0,1,1,0)$ & $(1,4,4,4)$ & $(\mathbf{0},1,\mathbf{1},\mathbf{2},4,4,\mathbf{3},4,\mathbf{4})$\\
		$(3,0,0,0)$ & $(4,3,3,4)$ & $(\mathbf{0},4,\mathbf{1},\mathbf{2},3,3,\mathbf{3},4,\mathbf{4})$\\
		$(3,0,1,0)$ & $(4,4,3,4)$ & $(\mathbf{0},4,\mathbf{1},\mathbf{2},4,3,\mathbf{3},4,\mathbf{4})$\\
		$(1,1,1,0)$ & $(2,4,4,4)$ & $(\mathbf{0},2,\mathbf{1},\mathbf{2},4,4,\mathbf{3},4,\mathbf{4})$\\
		$(3,1,1,0)$ & $(4,4,4,4)$ & $(\mathbf{0},4,\mathbf{1},\mathbf{2},4,4,\mathbf{3},4,\mathbf{4})$\\
	\end{tabular}
	\caption{Illustration of the map $\codetorvec$ for $\alpha=(1,2,1)$.}
	\label{tab:code_to_bracket_121}
\end{table}

\begin{proposition} \label{prop:bij-rvectocode}
Given a composition $\alpha$ of $n$, the transformation $\rvectocode$ is a
bijection from $\rvecset_\alpha$ to $\Codes_\alpha$.
\end{proposition}

\begin{proof}
First, for $\rvec \in \rvecset_\alpha$, let $\cvec = \rvectocode(\rvec)$
and let us check that $\cvec$ satisfies the conditions in
Definition~\ref{def:all_codes} using those in
Proposition~\ref{prop:reduced-def} for $\rvec$. By \eqref{it:rvec1} for
$\rvec$ and the definition of $\rvectocode$, clearly $\cvec$ satisfies
\eqref{it:code1}. Proposition~\ref{prop:reduced-decrease} and the definition
of $\rvectocode$ imply that $\cvec$ satisfies \eqref{it:code2}. To check
\eqref{it:code3} for $\cvec$ given that it satisfies \eqref{it:code2}, we only
need to show that, for any $i$ and $j$ such that $\indicator(i) <
\indicator(j)$, if $c_i \geq s_{\indicator(j)} - s_{\indicator(i)}$, then we
have $c_j + s_{\indicator(j)} \leq c_i + s_{\indicator(i)}$. Translating to
$\rvec$, we need to check that, for any $i'$ and $j'$ with $\indicator(i') <
\indicator(j')$, if $\rvec(i') \geq s_{\indicator(j')}$, then we have
$\rvec(j') \leq \rvec(i')$. Now, suppose that $\rvec(i') \geq
s_{\indicator(j')}$. We have $\indicator(\rvec(i')+1) > \indicator(j')$ by the
definition of $\indicator$. As the values are integers, we have
$\indicator(j') \leq \indicator(\rvec(i')+1) - 1$, which means $j' \leq
s_{\indicator(j')} \leq s_{\indicator(\rvec(i')+1)-1}$, and by
\eqref{it:rvec2}, we have $\rvec(j') \leq \rvec(i')$. Therefore, $\cvec$ also
satisfies \eqref{it:code3}.

Now for the reverse direction, given $\cvec \in \Codes_\alpha$, let $\rvec =
\codetorvec(\cvec)$. It is clear that \eqref{it:code1} translates directly to
\eqref{it:rvec1}. We only need to show that \eqref{it:rvec2} holds for
$\rvec$. Suppose that $1 \leq i < j \leq s_{\indicator(\rvec(i)+1)-1}$. If
$\indicator(i) = \indicator(j)$, by the definition of $\codetorvec$ and
\eqref{it:code2} on $\cvec$, we have $\rvec(j) \leq \rvec(i)$. Now we check
the case $\indicator(i) < \indicator(j)$. When translated to $\cvec$,
\eqref{it:rvec2} in this case means that we need to check for any $i' < j'$
such that $\indicator(i') < \indicator(j') \leq \indicator(c_{i'} +
s_{\indicator(i')} + 1) - 1$, we have $c_{j'} + s_{\indicator(j')} \leq c_{i'}
+ s_{\indicator(i')}$. By \eqref{it:code2}, we may assume that $j' = s_a$ for
some $a$. By the definition of $\indicator$, we see that $\indicator(s_a) \leq
\indicator(c_{i'} + s_{\indicator(i')}+1) - 1$ implies $s_a < c_{i'} +
s_{\indicator(i')} + 1$, thus $s_a \leq c_{i'} + s_{\indicator(i')}$ since
they are integers. By \eqref{it:code3}, we have $c_{s_a} \leq c_{i'} - s_a +
s_{\indicator(i')}$. Therefore, \eqref{it:rvec2} holds for $\rvec$, meaning
that $\rvec \in \rvecset_\alpha$.
\end{proof}

This allows us to conclude to the announced simple proof of
Theorem~\ref{thm:alpha_tamari_isomorphism}.

\begin{proof}[Proof of Theorem~\ref{thm:alpha_tamari_isomorphism}]
Let $\Tamari_{\nu_{\alpha}}$ denote the $\nu_{\alpha}$-Tamari lattice.  We
have the following isomorphisms of lattices:
	\begin{displaymath}
		\Tamari_{\alpha} \overset{\vphantom{\mathrm{Pp}}\mathrm{Thm}.~\ref{thm:alpha_tamari_lattice_codes}}{\cong} \bigl(\Codes_{\alpha},\comporder\bigr) \overset{\mathrm{Prop}.~\ref{prop:bij-rvectocode}}{\cong} \bigl(\rvecset_{\alpha},\comporder\bigr) \overset{\vphantom{\mathrm{Pp}}\mathrm{trivial}}{\cong} \bigl(\bvecset_{\alpha},\comporder\bigr) \overset{\vphantom{\mathrm{Pp}}\mathrm{Thm}.~\ref{thm:nu_tamari_bracket_vectors}}{\cong} \Tamari_{\nu_{\alpha}}.\qedhere
	\end{displaymath}
\end{proof}

Note that the proof of Theorem~\ref{thm:alpha_tamari_isomorphism} in
\cite{ceballos20the} relies on lattice-theoretic properties of
$\Tamari_{\alpha}$ and $\Tamari_{\nu_{\alpha}}$, and is only partially
bijective. Our proof here is fully bijective, which gives a clearer vision of
the isomorphism. 

\section{A combinatorial anti-isomorphism on the $\nu_{\alpha}$-Tamari lattice}
\label{sec:stack}

%

\subsection{Two ways from $(\alpha,231)$-avoiding permutations to
$\alpha$-paths}
\label{sec:two_bijections}

Recall that we have fixed a composition $\alpha = (\alpha_1, \alpha_2, \ldots,
\alpha_r)$ of $n$, and that $s_a = \alpha_1 + \alpha_2 + \cdots + \alpha_a$
for $a\in[r]$. Moreover, recall the definition of the \defn{$\alpha$-bounce
path} from \eqref{eq:alpha_bounce_path}.  We usually say \defn{$\alpha$-path}
rather than $\nu_{\alpha}$-path.

We now define two bijections from $(\alpha,231)$-avoiding permutations to
$\alpha$-paths.
The first one uses the $\alpha$-code from Section~\ref{sec:alpha_encoding},
and sends $w\in\Symmetric_{\alpha}(231)$ to
$\varphi(w)\in\Paths_{\nu_{\alpha}}$ satisfying
$\varphi(w)=[f_{0},f_{1},\ldots,f_{n}]$, where
\begin{equation}
\label{eq:fi-def}
f_i \defs \Bigl\lvert\bigl\{ j \mid 1 \leq j \leq n, \code_\alpha(w)_j +
s_{\indicator(j)} = i \bigr\}\Bigr\rvert.
\end{equation}
%
For example, for $\alpha=(1,3,1,2)$ and $w = \hspace*{-.25cm}\raisebox{-.12cm}{\AlphaPerm{1,3,1,2}{c1,c2,c3,c4}{5,3,4,7,1,2,6}{.3}{.25}}\hspace*{-.2cm}\in\Symmetric_{\alpha}(231)$,
we have $\code_{\alpha}(w)=(2,2,2,3,0,0,0)$ and $\varphi(w)=[0,0,0,1,0,1,2,3]$.
Note that the first entry of $\varphi(w)$ is always $0$.

\begin{proposition}
\label{prop:perm-to-path-code}
 The map $\varphi$ is an isomorphism from $\Tamari_\alpha$ to
$\Tamari_{\nu_\alpha}$.
\end{proposition}

\begin{proof}
We write $\varphi(w) = [f_0, f_1, \ldots, f_n]$. By the definition of
$\rvectocode$, we have
  \[
   f_i = \Bigl\lvert\bigr\{ j \mid 1 \leq j \leq n, \rvectocode^{-1}\bigl(\code_\alpha(w)\bigr)_j  = i \bigr\}\Bigr\rvert.
  \]
We may rephrase this using $\nu_{\alpha}$-reduced vectors. Let
$\rvec=\codetorvec\bigl(\code_{\alpha}(w)\bigr)$ denote the
$\nu_{\alpha}$-reduced bracket vector associated with $w$. Then,
  \[
    f_i = \Bigl\lvert\{ j \mid 1 \leq j \leq n, \rvec_j = i \bigr\}\Bigr\rvert - 1.
  \]
Now, if $\bvec=\extrvec(\rvec)$ is the associated $\nu_{\alpha}$-bracket
vector, then the number of entries equal to $i$ in $\bvec$ is $f_{i}+1$.
According to \cite{ceballos20nu}*{Definition~26}, there exists a unique
$\nu_{\alpha}$-path with as many lattice points of height $i$ as there are
entries equal to $i$ in $\bvec$.  (See also item (ii) in the proof of
\cite{ceballos20nu}*{Proposition~27}.)  We conclude that the
$\nu_{\alpha}$-path associated with $\bvec$ is precisely $\varphi(w)$.
  
Therefore, $\varphi$ is precisely the isomorphism used in the proof of
Theorem~\ref{thm:alpha_tamari_isomorphism}.
\end{proof}

The second bijection, denoted by $\Theta$, from $(\alpha, 231)$-avoiding
permutations to $\alpha$-paths was first defined in \cite{muehle19tamari}. We
will use an equivalent definition derived from \cite{ceballos20the}, using a
family of trees called \emph{$\alpha$-trees}, which we will not explicitly
define.

\begin{construction}\label{constr:theta}
Given $w \in \Symmetric_\alpha(231)$, we construct a labeled plane tree $T(w)$
by an insertion procedure. We start with a node labeled $n+1$ as the root, and
we read the elements of $w$ from left to right. Upon reading of an element
$w(i)$, we start a walk from the root. When we reach a node $v$ with label
$\ell$, if $w(i) < \ell$, then we move to the left-most child of $v$;
otherwise, we move to the first sibling of $v$ on its right. When the
destination node does not exist, we add it with label $w(i)$ and terminate the
walk. The labeled plane tree thus obtained is denoted by $T(w)$.

Now we construct a northeast path $P$ from $T(w)$. If the root of $T(w)$ has
$k$ children, then we start $P$ with $k$ north steps. Then, for each
$a\in[r]$, we inspect the elements $w(i)$ in the $a$-th $\alpha$-region
\emph{from right to left}, \ie $i$ runs from $s_a$ down to $s_{a-1} + 1$. For
each such $w(i)$, let $v_i$ be the node with label $w(i)$ in $T(w)$, and we
append $EN^{k_i}$ to $P$, where $k_i$ is the number of children of $v_i$. We
define $\Theta(w)$ to be the path $P$ thus obtained. See
Figure~\ref{fig:theta} for an example.
\end{construction}

\begin{figure}
  \centering
  \includegraphics[page=4,width=0.9\textwidth]{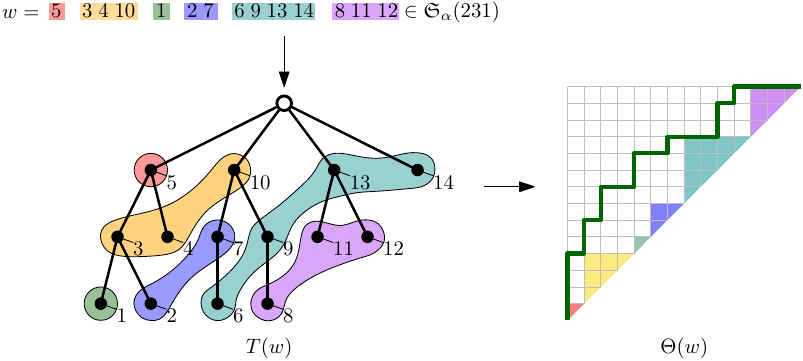}
  \caption{Example of $\Theta(w)$ for $w \in \Symmetric_\alpha(231)$ with
$\alpha = (1,3,1,2,4,3)$.}
  \label{fig:theta}
\end{figure}

For $w\in\Symmetric_{\alpha}(231)$, the tree $T(w)$ from
Construction~\ref{constr:theta} has the following immediate property.

\begin{proposition}[\cite{ceballos20the}]
\label{prop:lac-reading}
For $w \in \Symmetric_\alpha(231)$, let $T(w)$ be the labeled plane tree
constructed in Construction~\ref{constr:theta}. Reading the labels of $T(w)$
in postorder (\ie for each node $u$, the children of $u$ are increasing in
order from left to right, all greater than $u$) gives $1, 2, \ldots, n$.
\end{proposition}

\begin{remark} \label{rmk:theta}
The map $\Theta$ in Construction~\ref{constr:theta} is in fact $\Theta^{-1}
\circ \flip$ in \cite{ceballos20the}. We have altered the definition here for
simplicity. In Construction~\ref{constr:theta}, the tree $T(w)$ is a labeled
version of an $\alpha$-tree, and the map from $w$ to $T(w)$ is the map
$\Lambda_{\mathrm{perm}}$ in \cite{ceballos20the}. Moreover, the map from
$T(w)$ to $\Theta(w)$ is $\Xi_{\mathrm{nn}}$ in that article, but our
definition here is adapted from Lemma~1.31 of the same article. The validity
of our definition of $\Theta$ is ensured by
\cite{ceballos20the}*{Propositions~1.33~and~1.34}. The property in
Proposition~\ref{prop:lac-reading} follows from
\cite{ceballos20the}*{Construction~1.14}.
\end{remark}

\subsection{A stack-processing procedure}

We now give another combinatorial definition of $\varphi$. For $w \in
\Symmetric_\alpha(231)$, we define the \defn{companion} of an element $w(i)$
in $w$ to be the last element it sees, or $w(s_{\indicator(i)})$ when $w(i)$
sees no element. Then we can define $\varphi(w) = [f_0, f_1, \ldots, f_n]$,
where $f_i$ is the number of elements in $w$ with $w(i)$ as its companion. We
check that this definition of $f_i$ is the same as \eqref{eq:fi-def}. We
define the following \defn{stack processing} that can be used to compute the
companions of elements of $w$.

\begin{construction} \label{constr:stack}
Given $w \in \Symmetric_\alpha(231)$, we start with an empty stack $S$ and
then perform the following steps on the $\alpha$-regions \emph{in reverse
order}, \ie $k$ runs from $r$ down to $1$.
\begin{itemize}
\item (\textbf{Popping}) For $i$ from $1$ to $\alpha_k$, consider the $i$-th
element $w(s_{k-1}+i)$ in region $k$. Pop elements from the stack until the
top one is larger than $w(s_{k-1}+i)$. 
\item (\textbf{Pushing}) For $i$ from $1$ to $\alpha_k$, push the element
$w(s_k-i+1)$ into the stack.
\end{itemize}
There are $n$ elements in $w$, and each element passes through two steps,
totaling to $2n$ steps. See Figure~\ref{fig:stack} for an example.
\end{construction}

\begin{remark}
Note that in terms of popping elements we only need the popping
step for the last element in each region, as it is also the largest. However,
taking the popping step for each element into account is important to
understand the link between $\varphi$ and $\Theta$. Namely, given $w \in
\Symmetric_\alpha(231)$, the number of elements popped out in the popping step
of $w(i)$ is the number of children of the node with label $w(i)$ in $T(w)$
(Proposition~\ref{prop:stack-active}), which is in turn the length of the
corresponding run of north steps in $\Theta(w)$.
\end{remark}

\begin{figure} 
\centering
\includegraphics[page=6,width=0.8\textwidth]{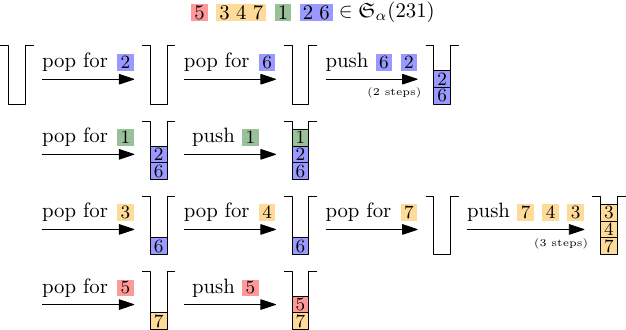}
\caption{An example of stack processing for $w = 5347126\in\Symmetric_{\alpha}(231)$ with $\alpha = (1,3,1,2)$.} \label{fig:stack}
\end{figure}

\begin{lemma}\label{lem:stack-prop}
For $w \in \Symmetric_\alpha(231)$, at each step of the stack processing of
$w$ with the stack $S$, we have:
\begin{enumerate}[\rm (i)]
\item the elements of $S$ are increasing from top to bottom;
\item for elements in $S$, their indices in $w$ are increasing from top to
bottom.
\end{enumerate}
\end{lemma}

\begin{proof}
For the first point, we proceed by induction on the number of steps. The
claim is clearly satisfied at the beginning, when $S$ is empty. The popping
step maintains the claim. Suppose that we are now pushing an element $w(i)$
into $S$. If $w(i)$ is the last element in its region, then by induction
hypothesis, all elements smaller than $w(i)$ should have been popped out;
otherwise, the top of the stack is $w(i+1) > w(i)$. In both cases, pushing
$w(i)$ maintains our claim. We thus conclude the induction.

For the second point, we observe that the claim is valid for the empty stack,
and the pushing steps maintain the claim, since all the elements in the same
region are pushed consecutively starting with the last element in the region.
The popping steps clearly also maintains the claim.
\end{proof}

\begin{proposition} \label{prop:stack-companion}
For $w \in \Symmetric_\alpha(231)$, we consider the popping step of an
element $w(i)$ in the stack processing of $w$ with stack $S$. If after that
popping step $S$ is not empty with top element $w(j)$, then the companion of
$w(i)$ is $w(j-1)$; if $S$ is empty, then the companion of $w(i)$ is $w(n)$.
\end{proposition}

\begin{proof}
Assume that there is an index $\ell$ such that $s_{\indicator(i)} < \ell < j$
and $w(\ell) > w(i)$. We take the smallest such index $\ell$. Then $w(\ell)$
cannot be popped out of $S$ before the treatment of the region
$\indicator(i)$, since an element $w(\ell')$ that pops $w(\ell)$ out must have
$s_{\indicator(i)} < \ell' < \ell$ and $w(\ell') > w(\ell) > w(i)$, violating
the minimality of $\ell$. At this moment, $w(j)$ is also in $S$. By
Lemma~\ref{lem:stack-prop}(ii), $w(j)$ is below $w(\ell)$. Since $w(j)$ is on
top of $S$ after the popping step of $w(i)$, there is some $w(i')$ in the same
region of $i$ with $i' < i$ that popped $w(\ell)$ out, meaning that $w(\ell) <
w(i') < w(i)$, contradicting our hypothesis. Therefore, such $w(\ell)$ does
not exist. From the definition of the popping step, we have $w(j) > w(i)$.
Thus, the companion of $w(i)$ is $w(j-1)$. In the case of empty stack, it
means that no element in previously inserted regions is larger than $w(i)$,
thus the companion of $w(i)$ is $w(n)$.
\end{proof}

\subsection{Stack processing and the bijection $\Theta$}

We now describe a link between stack processing of $w \in \Symmetric_\alpha(231)$ and the tree $T(w)$ in Construction~\ref{constr:theta}. Given $1 \leq k \leq r$, the \defn{nodes of region $k$} of $T(w)$ are those with labels corresponding to the values of $w$ in the $k$-th $\alpha$-region. We say that the root is of region $0$, and that the region $r+1$ is empty. From the insertion procedure, for a node of region $k$ with its parent of region $k'$, we have $k > k'$. The \defn{active nodes} for region $k$ are nodes of region $k' \geq k$ whose parent is of region $k'' < k$.

\begin{proposition}\label{prop:stack-active}
  For $w \in \Symmetric_\alpha(231)$, consider the stack processing of $w$ with stack $S$. For $1 \leq k \leq r+1$, the elements in $S$ after processing 
	the $k$-th $\alpha$-region
  are exactly the labels of the active nodes for region $k$ in $T(w)$.

  Furthermore, for a node $u$ in $T(w)$, the labels of its children are exactly the elements popped out by the label of $u$ in the stack processing of $w$.
\end{proposition}
\begin{proof}
  Let $E_k$ be the set of elements in $S$ after processing the $k$-th $\alpha$-region, and let $L_k$ be the set of labels of the active nodes of region $k$ in $T(w)$. We show that $E_k = L_k$. We proceed by induction on $k$ from $r+1$ to $1$. For $k = r+1$, the set $E_{r+1}$ is empty, and there are no active nodes for region $r+1$, so we have $E_{r+1} = L_{r+1}$. Suppose that $E_{k+1} = L_{k+1}$. We observe that, by Construction~\ref{constr:stack} and Lemma~\ref{lem:stack-prop}(i), we have
  \begin{equation} \label{eq:stack-active-1}
    E_k = \left( E_{k+1} \setminus R_k \right) \cup \{ w(i) \mid \indicator(i) = k \},
  \end{equation}
  where $R_k = \{ w(i) \mid \indicator(i) > k, w(i) < w(s_k) \}$.

  Now, by the definition of active nodes, we split $L_k$ into two parts, $L_k^{(1)}$ for nodes of region $k$, and $L_k^{(2)}$ for other nodes. It is clear that $L_k^{(1)} = \{ w(i) \mid \indicator(i) = k \}$. A node with label in $L_k^{(2)}$ must be in some region $k' > k$, and its parent in some other region $k'' < k < k+1$. Therefore, $L_k^{(2)} \subseteq L_{k+1}$. Conversely, let $u$ be a node with labels in $L_{k+1}$, and let $v$ be the parent of $u$ in region $k_v$. If $k_v < k$, then $u$ is also in $L_k^{(2)}$; otherwise, if $k_v = k$, then $u$ is not in $L_k^{(2)}$. We thus have
  \begin{equation} \label{eq:stack-active-2}
    L_k = (L_{k+1} \setminus R_k') \cup \{ w(i) \mid \indicator(i) = k \}.
  \end{equation}
  Here, $R_k'$ is the set of labels of nodes whose parents are of region $k$. 

  Now, by the construction of $T(w)$ in Construction~\ref{constr:theta}, labels in $R_k'$ must be in some region $k' > k$, and they are smaller than some element in region $k$ of $w$, therefore smaller than $w(s_k)$. We thus have $R_k' \subseteq R_k$. Conversely, suppose that $R_k \setminus R_k'$ is not empty, and let $w(i) \in R_k \setminus R_k'$ and $u$ the node in $T(w)$ with $w(i)$ as label. We have $\indicator(i) > k$ and $w(i) < w(s_k)$. Let $v$ be the parent of $u$ in $T(w)$, and $w(j)$ the label of $v$. As $w(i) \in R_k'$, we know that $v$ is in some region $k_v < k$. Suppose that $v'$ is the node with label $w(s_k)$. As $w(i) < w(s_k)$, by Proposition~\ref{prop:lac-reading}, $u$ precedes $v'$ in postorder. By the construction of $T(w)$, we know that the region of a node is always strictly smaller than that of its children, and weakly smaller than that of its siblings to the right. Therefore, the node $v$ of region $k_v < k$ must not be a descendant of $v'$ of region $k$. If $v'$ is a descendant of $v$, as $u$ precedes $v'$ in postorder, meaning that $v'$ must be a sibling of $u$ to the right, or a descendant of such a sibling, which is impossible because $u$ is in region $\indicator(i) > k$. Therefore, $v'$ and $v$ are not comparable in $T(w)$, and along with the fact that $u$ precedes $v'$ in postorder, $v$ also precedes $v'$ in postorder. Now take the rightmost child of $v$, say $u'$ with label $w(i')$, which also precedes $v'$ in postorder. We have $w(j) = w(i')+1$ and $w(s_k) > w(j)$ by Proposition~\ref{prop:lac-reading}. Furthermore, $v$ is in region $k_v < k$, while $v'$ is in region $k$ and $u'$ is in a region $k_{u'} \geq \indicator(i) > k$. We thus have an $(\alpha,231)$-pattern $w(j), w(s_k), w(i')$ in $w$, which is not possible. Therefore, $w(i)$ cannot exist, and we have $R_k' = R_k$.

  Comparing \eqref{eq:stack-active-1} and \eqref{eq:stack-active-2}, along with $R_k' = R_k$ and the induction hypothesis $E_{k+1} = L_{k+1}$, we have $E_k = L_k$, concluding the induction. Therefore, the first part of our claim holds for all $1 \leq k \leq r+1$.

  For the second part, let $w(i)$ be the label of $u$ and $w(j)$ an element popped out by $w(i)$ in the stack processing, and $v$ the node with $w(j)$ as label. Suppose that $u$ is of region $k_u$. By the first part of our claim, $v$ is an active node for region $k_u - 1$ but not for region $k_u$. Therefore, the parent of $v$ is of region $k_u$. By Proposition~\ref{prop:lac-reading}, the label of the parent of $v$ must be the first element in region $k_u$ larger than $w(j)$, which is $w(i)$ according to the popping step of region $k_u$. Thus, $u$ is the parent of $v$, and we have the second part of our claim.
\end{proof}

We now prove that the isomorphism $\varphi$ from $\Tamari_\alpha$ to
$\Tamari_{v_\alpha}$ is closely related to $\Theta$ defined in
\cite{ceballos20the}; see Section~\ref{sec:two_bijections}

\begin{theorem} \label{thm:phi-theta}
  Let $\alpha=(\alpha_{1},\alpha_{2},\ldots,\alpha_{r})$.  For $w \in \Symmetric_\alpha(231)$, we have
  \[
    \bigl(\nu_\alpha, \varphi(w)\bigr) = \conj'\Bigl(\flip(\nu_\alpha), \flip\bigl(\Theta(w)\bigr)\Bigr).
  \]
\end{theorem}
\begin{proof}
  By definition of $\conj'$ and Proposition~\ref{prop:drun-cont}, we only need to show that 
  \begin{displaymath}
	\drun\bigl(\nu_\alpha, \varphi(w)\bigr) = \cont\Bigl(\flip(\nu_\alpha), \flip\bigl(\Theta(w)\bigr)\Bigr),
  \end{displaymath}
	as $\drun$ is injective. For the pair $\bigl(\nu_\alpha, \varphi(w)\bigr)$, we observe that
  \[
    \nu_\alpha = [0, 0^{\alpha_1 - 1}, \alpha_1, 0^{\alpha_2 - 1}, \alpha_2, \ldots, 0^{\alpha_r - 1}, \alpha_r],
  \]
  with $0^k$ standing for $k$ entries of $0$. Suppose that $\varphi(w) = [f_0, f_1, \ldots, f_n]$. We have
  \begin{align}
    \begin{split}\label{eq:phi-theta-drun}    
      \drun\bigl(\nu_\alpha, \varphi(w)\bigr) &= (f_{s_r} + 1, 0^{\alpha_r}, f_{s_r-1} + 1, f_{s_r-2} + 1, \ldots, f_{s_{r-1}+1} + 1, 0^{\alpha_{r-1}}, \ldots, \\
      &\quad\quad\quad\quad\quad f_{s_1-1} + 1, \ldots, f_0 + 1, 0^{\alpha_1}, 0).
    \end{split}
  \end{align}

  Now for the pair $\bigl(\nu_\alpha, \Theta(w)\bigr)$, for $0 \leq i \leq 2n$, let $Q_i'$ the $(i+1)$-st lattice point $Q_i'$ on $\nu_\alpha$ \emph{in reverse order}. We define $d'_i$ to be the number of north steps $(0,1)$ we can take from $Q_i'$ without crossing to the other side of $\Theta(w)$. It is clear that $d'_i$ is also the reverse horizontal distance of the $(i+1)$-st lattice point of $\flip(\nu_\alpha)$ with respect to $\flip\bigl(\Theta(w)\bigr)$. By Proposition~\ref{prop:cont-path-pair},
  \begin{align*}
    \cont\Bigl(\flip(\nu_\alpha), \flip\bigl(\Theta(w)\bigr)\Bigr)_0 &= \bigl\lvert\{ \ell \mid 1 \leq \ell \leq 2n, d'_\ell = 0 \}\bigr\rvert; \\
    \cont\Bigl(\flip(\nu_\alpha), \flip\bigl(\Theta(w)\bigr)\Bigr)_i &= \bigl\lvert\{ \ell \mid i < \ell \leq 2n, d'_\ell = d'_i + 1, \forall i < m \leq \ell, d'_m > d'_i \}\bigr\rvert.
  \end{align*}

Consider the stack processing of $w$ with stack $S$. We now show that the
number of elements in the stack after $i$ steps of stack processing is $d_i'$.
We proceed by induction on the number of steps we have taken in the stack
processing. In the initial stage, $d'_0 = 0$ agrees with the empty stack. When
dealing with region $k$, we first perform the popping step. By the
construction of $\Theta(w)$, for the $i$-th element in region $k$, the number
of children of its correspondent node, which is also the number elements
popped out by $w(s_{k-1}+i)$ by Proposition~\ref{prop:stack-active}, is the
number of north steps of $\Theta(w)$ on abscissa $s_k-i+1$, which is exactly
$d'_{2n-2s_k+i-1} - d'_{2n-2s_k+i}$. Then for the pushing step, the stack size
increases by $1$ at each step, just as when we pass from $d'_{2n-s_k+i-1}$ to
$d'_{2n-s_k+i}$. We thus conclude the induction.

We now show that $\drun\bigl(\nu_\alpha, \varphi(w)\bigr) =
\cont\Bigl(\flip(\nu_\alpha), \flip\bigl(\Theta(w)\bigr)\Bigr)$. First, we
know that $d'_i$ is weakly decreasing for $i$ from $2(n-s_k)+1$ to
$2(n-s_k)+\alpha_k$ for all $1 \leq k \leq r$, and by definition,
$\cont\Bigl(\flip(\nu_\alpha), \flip\bigl(\Theta(w)\bigr)\Bigr)$ takes the
form
  \[
    (g_{s_r}, 0^{\alpha_r}, g_{s_r-1}, g_{s_r-2}, \ldots, g_{s_{r-1}}, 0^{\alpha_{r-1}}, \ldots, g_{s_1-1}, \ldots, g_0, 0^{\alpha_1}, 0).
  \]
Here, $(g_i)_{0 \leq i \leq n}$ is a sequence of positive integers. The last
$0$ is from the last point, because it does not have any lattice point after
it. In comparison to $\drun\bigl(\nu_\alpha, \varphi(w)\bigr)$, it is clear
that we only need to prove $g_\ell = f_\ell + 1$ for all $\ell$.

For $\ell=n$, according to Proposition~\ref{prop:stack-companion}, an element
$w(j)$ has $w(n)$ as its companion if, and only if the stack is empty after
its popping step, which is equivalent to $d'_j = 0$. Therefore, $g_n = f_n +
1$. For $\ell = 0$, it is clear that $g_0 = 1 = f_0 + 1$, since $\nu_\alpha$
starts with a north step.

For $0 < \ell < n$, we know that $f_\ell$ is the number of nodes with
$w(\ell)$ as companion, which is also the number of times we see $w(\ell+1)$
at the top of the stack during a popping step according to
Proposition~\ref{prop:stack-companion}. Suppose that $w(\ell+1)$ is the $i$-th
element in region $k$, thus $\ell + 1 = s_{k-1} + i$. We know that $w(\ell+1)$
is pushed into the stack at step $2(n-s_{k-1})-i+1$. Suppose that there are $p
= d'_{2(n-s_{k-1})-i}$ elements before $w(\ell+1)$ is pushed down, we have
$d'_{2(n-s_{k-1})-i+1} = p + 1$. Then $w(\ell+1)$ is popped out once $d'_j
\leq p$. When we see $w(\ell+1)$ on top of the stack, we must have $d'_j =
p+1$ before it is popped. This is exactly the definition of
$\cont(\flip(\nu_\alpha), \flip(\Theta_1(w)))_{2(n-s_{k-1})-i}$, which is also
$g_{s_{k-1} + i - 1}$. We thus know that $g_{s_{k-1} + i - 1}$ is the number
of times we see $w(\ell+1)$ on the top of the stack, the first time it is
pushed, the other times we have an element whose companion is $w(\ell)$. We
thus have $g_\ell = g_{s_{k-1}+i-1} = f_{\ell} + 1$. It follows that
$\drun(\nu_\alpha, \varphi(w)) = \cont(\flip(\nu_\alpha), \flip(\Theta(w)))$,
which concludes the proof.
\end{proof}

\begin{figure} 
\centering
\includegraphics[page=2,width=\textwidth]{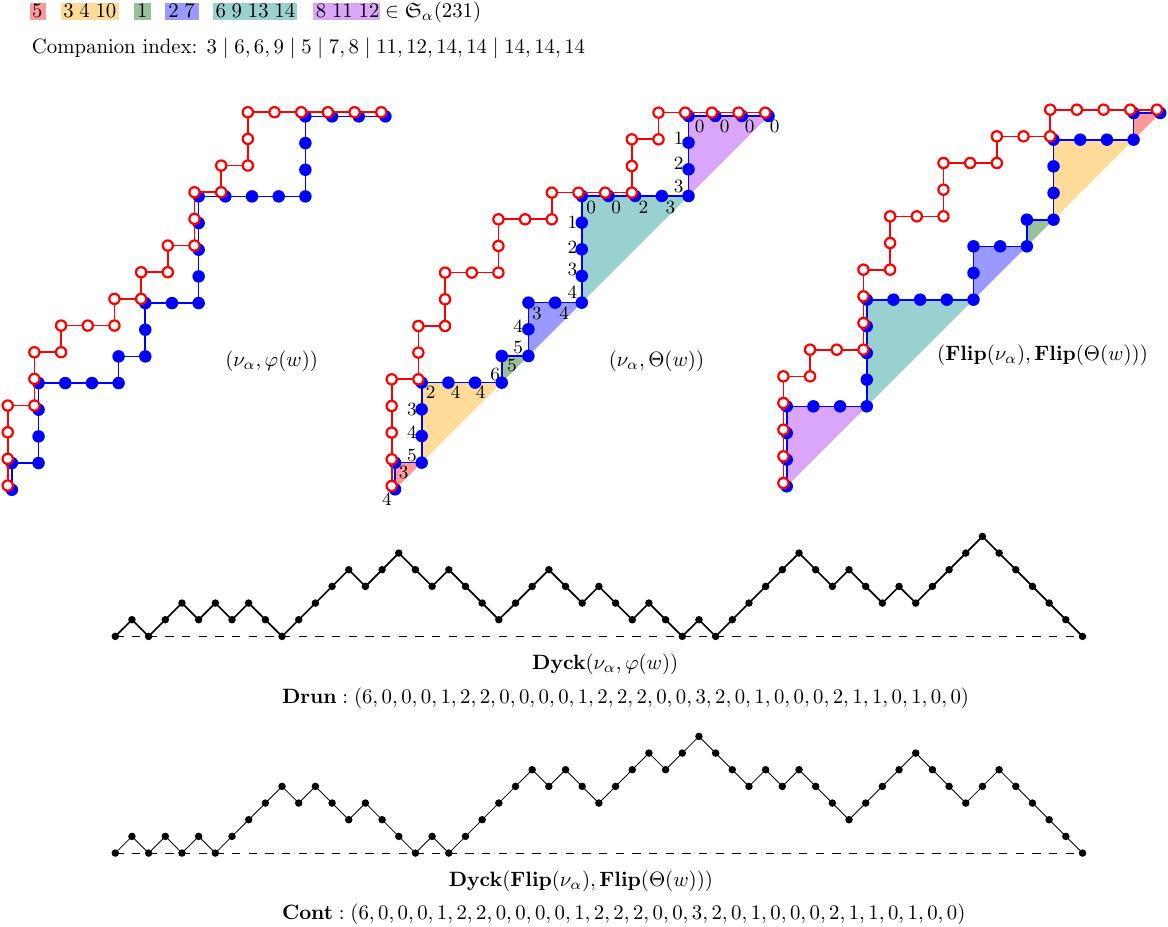}
\caption{An example of the transfer of statistics in the two ways of looking at parabolic objects.} \label{fig:conjbijproof}
\end{figure}

An example of the proof of Theorem~\ref{thm:phi-theta} can be seen in Figure~\ref{fig:conjbijproof}. We thus solve \cite{ceballos20the}*{Open Problem~2.23}.

\begin{corollary}\label{coro:phi-theta}
  The map $\flip \circ \Theta$ is an anti-isomorphism between $\Tamari_\alpha$ and $\Tamari_{\flip(\nu_\alpha)}$.
\end{corollary}
\begin{proof}
  This is a consequence of Theorem~\ref{thm:phi-theta}, and Propositions~\ref{coro:tamari-inv}~and~\ref{prop:perm-to-path-code}.
\end{proof}

\begin{remark}\label{rmk:typo-open-problem}
There is a typo in \cite{ceballos20the}*{Open Problem~2.23}. It should be
``lattice anti-isomorphism'' instead of ``lattice isomorphism'', as we can
also see in Figure~11 therein.
\end{remark}

\end{document}